\documentclass[12pt]{amsart}
\usepackage{geometry}                
\usepackage{graphicx}
\usepackage{amsthm}
\usepackage{amsmath} 
\usepackage{amssymb}
\usepackage{epstopdf}

\usepackage[unicode=true,pdfusetitle,
 bookmarks=true,bookmarksnumbered=false,bookmarksopen=false,
 breaklinks=false,pdfborder={0 0 1},backref=section,colorlinks=true]
 {hyperref}

\newtheorem{thm}{Theorem}
\newtheorem{lem}{Lemma}
\newtheorem{cor}[thm]{Corollary}

\newtheorem{defn}{Definition}
\theoremstyle{remark}
\newtheorem*{remark}{Remark}

\newcommand{\dee}{\mathrm{d}}

\newcommand{\R}{\mathbb{R}}
\newcommand{\T}{\mathbb{T}}
\newcommand{\Z}{\mathbb{Z}}

\DeclareMathOperator{\K}{K}

\title{Syzygies in the  two   center problem}
\author{Holger R.~Dullin}
\address{University of Sydney\\
School of Mathematics and Statistics\\
Sydney, NSW 2006\\
Australia}
\email{holger.dullin@sydney.edu.au}
\author{Richard Montgomery}
\address{UC Santa Cruz\\
Department of Mathematics\\
4111 McHenry\\
Santa Cruz, CA 95064\\
USA}
\email{rmont@ucsc.edu}

\begin{document}

\date{Sep 14, 2015}
\maketitle

\begin{abstract}
We give a complete symbolic dynamics description of the dynamics of Euler's problem of two fixed centers. 
By analogy with the 3-body problem we use the collinearities (or syzygies) of the three bodies as symbols. 
We show that motion without collision on regular tori of the regularised integrable system are given by 
so called Sturmian sequences. Sturmian sequences were introduced by Morse and Hedlund in 1940.
Our main theorem is that the periodic Sturmian sequences are in one to one correspondence with 
the periodic orbits of the two center problem. Similarly, finite Sturmian sequences correspond to 
collision-collision orbits. 
\end{abstract}

%
%
%
%
%

\section{Introduction}

Associating a symbolic dynamics   to a smooth flow  goes back at least to Hadamard 
\cite{Hadamard1898}  with serious contributions
due to Morse and Hedlund \cite{MorseHedlund38, MorseHedlund40}.  The advent of Smale's Horseshoe \cite{Smale63}  and  Anosov flows \cite{Anosov67} made symbolic dynamics into
an  area of study connected to  smooth dynamical systems in a central way. 
In the area of  celestial mechanics, a number of partial results have been obtained about a tentative symbolic dynamics for  the planar three-body problem. See for example
\cite{Mont_action},  \cite{Mont_InfinitelyMany}, and \cite{MM2}. 
The symbols of this symbolic dynamics are the ``letters'' 1, 2 and 3 with letter $i$ representing an instant during the motion
where the three bodies have become collinear with body $i$ in the middle.   For historical reasons we call the resulting sequences ``syzygy sequences''.   
A big open question is, for the planar Newtonian three-body 
problem ``what are the possible syzygy  sequences  realized by motions of the three bodies?''   
Here, we answer this question for a much simpler   integrable system:
the planar two-center problem as investigated by Euler \cite{Euler1760}.  
Classical reference on this system are Charlier \cite{Charlier02} and Whittaker \cite{Whittaker37}, and a modern treatment is given in \cite{WDR03}.
The symbols here have precisely
the same meaning as for the full planar three-body problem but the system is completely integrable, hence we can completely solve the question. 
We hope  that this full solution in the integrable limit may shed   light on the honest three-body problem. 

\section{Symbolic dynamics}

Consider a  flow $\phi^t: M \to M$ on a manifold  $M$.
Fix a finite  collection of co-dimension one surfaces $S_i \subset M$, $i = 1, \dots, n$ 
(possibly with boundary) which we call   {\em windows}. 
 Attach   symbol $s_i$ to window $S_i$.
Associate  a   finite symbol sequence  to each  finite orbit segment  by 
   recording   the windows  in the order in which they are hit. 
  Thus  orbit segment  $\phi^t(x_0), 0 \le t \le T$ has  symbol sequence   $s_{i_1},s_{i_2}, \ldots s_{i_N}$.
 provided  $\phi^{t_a}(x_0) \in S_{i_a}$, $a =1, \ldots , N$ where $0 \le t_1 < t_2 < \ldots t_N \le T$ are the times at which the orbit hits
 {\it some window}. (The sequence is empty if no window is hit.) 
 This finite symbol   sequence is unique  provided the orbit
 avoids the  intersections of the  windows. The sequence is   stable with respect to small deformations
 of the initial condition $x_0$ provided the   orbit $\phi^t (x_0)$ is transverse to the windows.  
Letting $T \to \infty$ generally gives an infinite sequence. Running time backwards gives a bi-infinite sequence of symbols.

\begin{figure}[thb]
\includegraphics[]{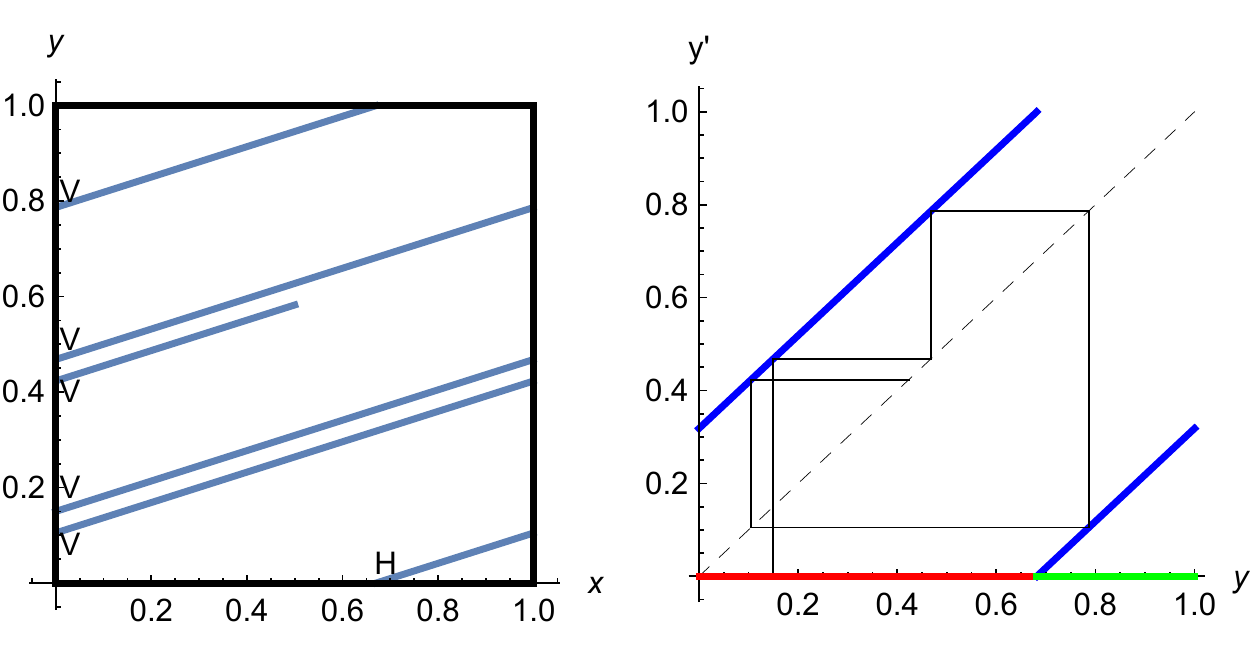}
\caption{Left: Linear flow on the torus with initial condition $(0, 0.15)$ and slope $W = 1/\pi$. 
The symbol sequence of the shown orbit segment is $VVVHVV$.
Right: Corresponding Poincar\'e map from $x=0$ to $x=1\equiv0$ 
giving the exponents $n_i = 0,0,1,0$. 
} \label{fig:Pmap}
\end{figure}

\subsection{Symbolic dynamics for integrable systems}
\label{Sturm}

Suppose now that our flow on $M$ is a  Liouville integrable system with $2$ degrees of freedom.
Then  $M$ is  a symplectic manifold $M$ of dimension $4$ endowed with  $2$ 
smooth  functions $G,H: M \to \R $ in involution,  independent almost everywhere,
one of which, $H$, we take to be the energy of the system.    
The pair  $F = (G,H) : M \to \R^2$ is called the integral map.
The Liouville-Arnold theorem states that if $c$ is a regular value of the integral map, 
then every compact connected component of the pre image $F^{-1}(c)$ is a $2$-torus.
Denote one of these tori by $\T_c$. Then in a neighbourhood $U$ of $\T_c$
one can introduce action-angle variables $(I_1, I_2,  \theta_1, \theta_2)$,  
such that the dynamics is given by 
\[
   \dot \theta_j = \omega_j(I_1, I_2), \quad \dot I_j = 0, \qquad j = 1, 2  \,.
\]
The tori are labelled by the value of the action $I = (I_1, I_2)$.
On each torus the   dynamics is  straight line  motion when
viewed from the covering space $\R^2$ of that torus: $\theta_i(t) = \theta_i(0) + \omega_i  (I)t$.
When changing from one torus to another,   $I$ and in general $\omega_i(I)$ will change, 
and hence  the slope of the lines changes. 

If   the windows $S_i \subset M$ intersect the torus $\T_c$ 
transversally then they  form  a system of windows      $\T_c \cap S_i$ (possibly empty)  for the restricted  linear dynamics on $\T_c$.
In our situation of the planar two-center problem the   windows will be   horizontal or vertical
lines: $\theta_1 = const$ or $\theta_2 = const$.    With this in  mind, the following example which goes back 
to Morse and Hedlund \cite{MorseHedlund40} is central to our work.

{\bf Example}: Consider a  linear flow on $\T^2 = \R^2 \mod \Z^2$ 
with positive  slope $m = \omega_2/\omega_1$. As  windows take the two  
two circles obtained by projecting the   x-axis and the y-axis of $\R^2$
onto the torus. Denote their symbols `H' for horizontal   and `V' for vertical.
Lifted to the $xy$ plane, the windows yield  the lines of    graph paper
intersecting at the lattice points.    The lifted orbits are the lines 
\begin{equation}
\label{line}
y = mx + b.
\end{equation}
Every  such orbit has a unique
 symbol sequence 
 as long as it stays away from lattice points.

The ratio of $H$'s to $V$'s in the symbol sequence of an   orbit  tends to the slope $m$
as the length of the orbit tends to infinity.   
The precise  sequence of $H$'s and $V$'s   
depends on the initial condition $y_0$, i.e., on   $b$. 
The resulting sequences are called ``Sturnian sequences'' 
after  Morse and Hedlund  coined this term in \cite{MorseHedlund40}.    
To describe this sequence let 
$\lfloor y \rfloor \in \Z$
denote the   greatest integer less than or equal to $y$, for $y$ a real number.
A bit of experimentation drawing lines on graph paper  suggests  that 
the associated symbol sequence always has the form
 \begin{equation}
\label{rotSeq_A}
 \ldots V H^{n_1} V H^ {n_2} V H ^{n_3} \ldots
 \end{equation}
where 
each nonnegative integer $n_i$ is either $\lfloor m \rfloor $ or $ \lfloor m \rfloor + 1$.  
Which  one though? 
To describe the $n_i$,   let 
$\{y \} $ denote the  fractional part of a real number $y$ so that $y = \lfloor y \rfloor + \{y \} $
with 
 $0 \le \{y\} <1 $.  Use  the Poincare map associated to
cutting across the vertical circle:  $y \mapsto y + m$ so as to obtain
a sequence of points $y_i \mod 1$ on the circle,  these being the orbit of $y_0$ uder the Poincare map, and their   associated fractional  parts.  
$\{y_k \} = \{ y_0 + k m \}$.  Then  
 \begin{equation}
\label{rotSeq}
   n_k = 
\begin{cases}
     \lfloor m \rfloor ,& \text{if } \{y_{k-1}\} + \{m \} <1 \\
    \lfloor m \rfloor + 1,& \text{if } \{y_{k-1}\} + \{m \}  >1\\        
\end{cases} \,.
\end{equation}
Here we have started with the point at $(0,y_0)$ and marked the initial `$V$'  in equation \eqref{rotSeq_A} as corresponding to ``$s_0$''
accordingly.
An equivalent  formula for the exponents  is $n_k =  \lfloor m + \{y_{k-1}\}  \rfloor$.

\begin{defn}  
\label{Sturmword}The sequence of  integers $\{n_k \}_{k \in \Z}$ defined by equation \eqref{rotSeq}, or the corresponding symbol sequence of $H$'s and $V$'s
will be called  the ``Sturmian sequence for rotation number $m$ and intercept $y_0$'',
assuming that the line $y = m x + y_0$ does not hit any lattice points.
\end{defn}

\begin{lem} \label{lem:rational} Suppose that the slope $m$ is rational. Then when  viewed as a periodic word, the Sturmian sequence associated to
rotation number $m$ and intercept $y_0$ is independent of the choice of intercept.  Thus we can speak of ``the Sturmian word
for $m$''  in case $m$ is rational.
\end{lem}
{\sc Proof.} Suppose   that $m = p/q$ is rational, with $p, q$ relatively prime.  Then the  initial conditions $(0,b)$, $0 < b < 1$ which pass thru the lattice points
are those for which $b = i/q, i =1, 2, \ldots, q-1$ and they cut the basic interval into $q$ equal intervals. Within any one of these intervals
the sequence is constant, since the associated line $y = mx + b$ can be translated without hitting lattice points.   
What about when we move our initial condition $(0,b)$ to a different interval?   
Use the lattice translations to decompose all the vertical unit intervals between vertically
adjacent  lattice
points into these $q$ equal segments.  Because $p$ and $q$ are relatively prime, any line starting in the interval $0 < y < 1/q$ eventually
hits all of the $q$ other segments.  We can thus imagine starting at any one of these $q$ segments, but at a later time
along the line, so arriving at a shifted version of the same sequence which corresponds to $b$ lying in this other segment. QED

 \begin{remark}  
 When $m$ is irrational the Sturmian word   depends on the intercept $y_0$ in an interesting and nontrivial manner.
\end{remark}

\begin{remark}
Sturmian sequences  arise in  billiards \cite{DB98}, in    the Henon map \cite{DMS98c}, and   in twist maps of the cylinder \cite{AubryLeDaeron83}.
\end{remark}


\begin{figure}
\includegraphics[]{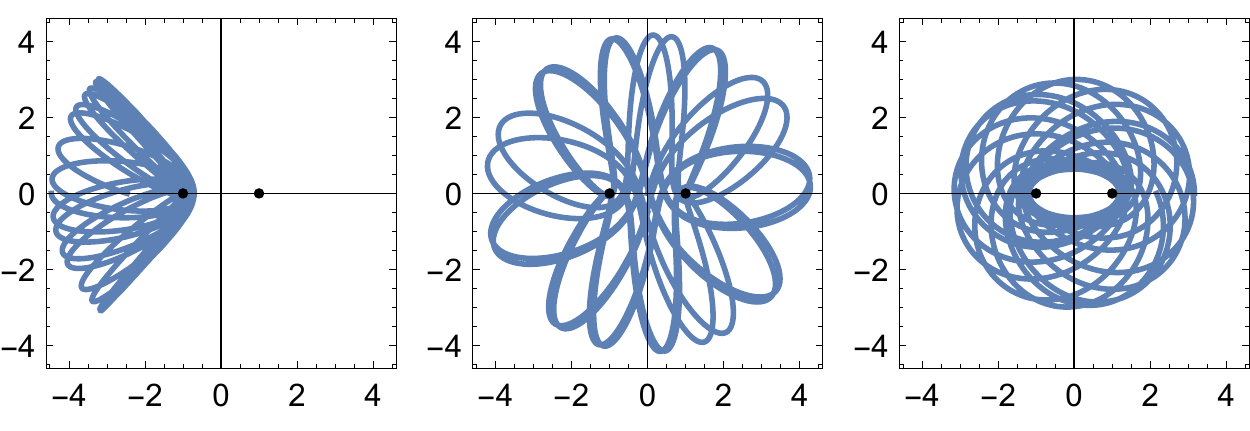}
\caption{Orbits in the $(x,y)$-plane for the symmetric case $m_1 = m_2 = 1/2$, $d=1$, and $h = -0.23$.
Type {\em Satellite} (left) intersects the $x$-axis near one of the centers only. 
Type {\em Lemniscate} (middle) intersects the $x$-axis near both centers. 
Type {\em Planetary} (right) does not intersect the segment between the centers. 
} \label{fig:caustics} 
\end{figure}

\section{The Two Center Problem and its tori.}

 We consider a unit test particle moving in the $xy$ plane $\R^2$ attracted by two fixed masses of mass $m_1$ and $m_2$  called ``centers'' and  located at the points $(-d, 0)$ and $(d,0)$.
The Hamiltonian in canonical coordinates $x, y, p_x, p_y$ with symplectic form $dx\wedge dp_x + dy \wedge dp_y$ 
is 
\begin{equation}
     H = \frac12( p_x^2 + p_y^2)  - \frac{m_1}{\sqrt{ (x+d)^2 + y^2}} -  \frac{m_2}{\sqrt{ (x-d)^2 + y^2}}.
\label{energy}
\end{equation}
Euler \cite{Euler1760} proved that the system is integrable with independent second integral 
\begin{equation}
     G  =\frac12 ( x p_y - y p_x)^2 + \frac12 d^2 p_x^2 + d x \left( \frac{m_1   }{\sqrt{(x+d)^2 + y^2}} - \frac{m_2 }{\sqrt{ ( x - d)^2 + y^2 } }\right) \,.
\label{G}
\end{equation}
$(G,H)$  form our    ``integral map'',   a map from phase space 
to the space of values of these two integrals.

Integration involves switching to regularizing variables for configuration space which are  denoted $\lambda, \nu$ below. 
For details see section \ref{sec:Euler} below.
These variables  have the  local effect of a Levi-Civita regularization about each of the two centers.  

\begin{defn}
\label{reg_torus} By a regular torus $\T$ we mean a compact connected component of an inverse image of a regular 
value of the integral map, expressed in regularizing variables.   By a regular periodic orbit we mean a periodic orbit without collisions 
lying on a regular torus. 
\end{defn}
Excluding the  critical values of the integral map,  there are precisely three topological types
of tori, denoted by S, L and P.  The projection of a torus of each type
onto configuration space are as shown in figure~\ref{fig:caustics}.
After regularization we can define ``action-angle'' variables
with angles  $\theta_1$, $\theta_2$.
These angles identify a regular torus  with $\R^2/ \Z^2$.  
We will say a torus, expressed in these variables,
has been ``flattened''.

\subsection{Windows, Syzygies} 
The phase space is $T^*( \R^2 \setminus \Delta)$ where $\Delta = \{ (-d,0), (d, 0) \}$
are the collisions. 
When the test  particle crosses the   $x$-axis the three masses become collinear. 
The centers  divide the $x$-axis into three intervals,  labelled ``1'' for the right infinite  interval $(-\infty, d)$,
``2'' for for the left infinite interval $(d, \infty)$. and  
``3'' for the central finite interval $-d < x < d$.  
This labelling corresponds precisely to the syzygy sequence labelling of the references \cite{Mont_InfinitelyMany} and \cite{MM2} mentioned above
provided  the test particle is labelled ``3'',  the right center as   ``1'' and left center  as ``2''.

The inverse image of the collinear locus, with its three windows, drawn
on the flattened torus  is shown for each type  in figure~\ref{fig:windows}. 
These pictures are at the heart of our analysis. 
We continue to denote the  inverse image of our syzygy  windows 1, 2, and 3,  under the regularization map, intersected with the torus,
and flattened by the symbols  1, 2, or 3. These windows are  a  collection  of  orthogonal lines.
How did a decomposition of   a straight line -- the $x$-axis -- into intervals   become a system of orthogonal lines? 
A regularization, qualitatively similar to the   the Levi-Civita regularization,
is needed to write down the action angle variables and hence express the torus.
At the  core of   Levi-Civita regularization is the squaring map,  the conformal map $\xi + i \eta \mapsto (x+ i y) = (\xi + i \eta)^2$
under which the inverse image of a straight line through the origin becomes a ``cross'': a pair of orthogonal lines,
and this fact explains how the $x$-axis turned into a collection of orthogonal lines.  See section \ref{sec:Euler} for details.

\begin{figure}
\includegraphics[width=15cm]{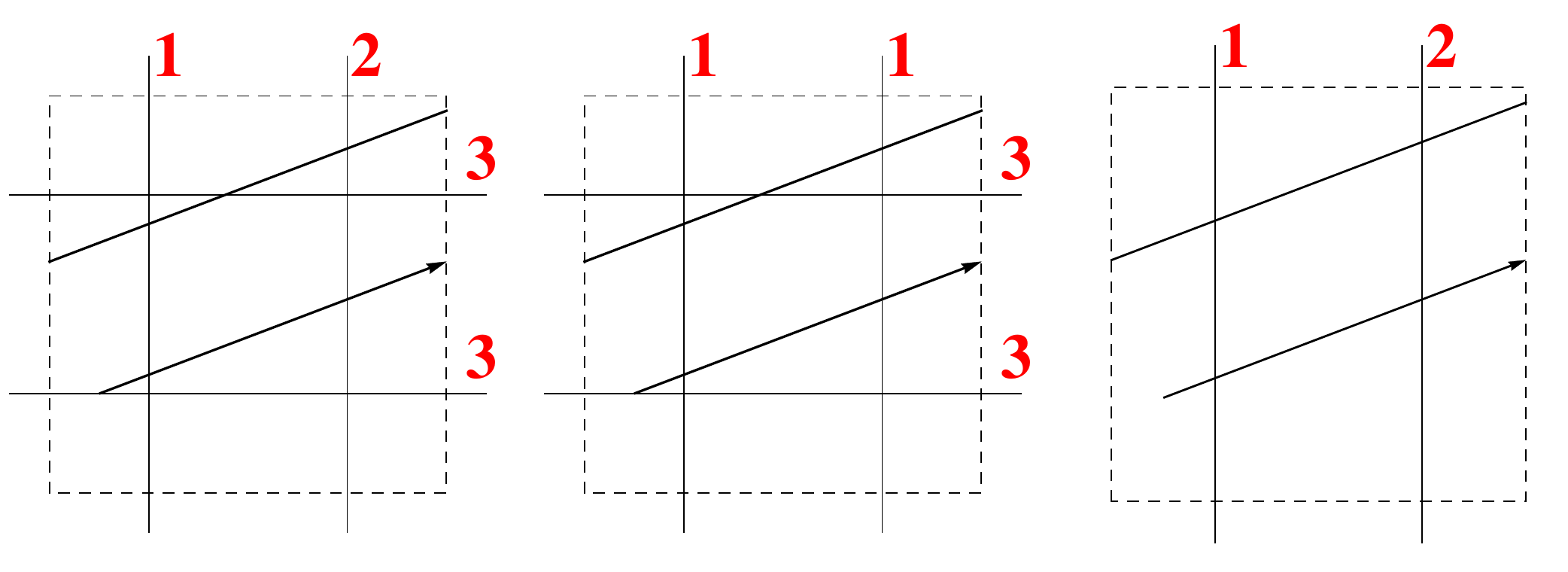}
\caption{The torus (dashed) in angle variables $\theta_1, \theta_2$ 
with a trajectory and windows with corresponding symbols  for type L
(Lemniscate), S, (Satellite), P (Planetary). 
The corresponding symbol sequences for the orbit segment shown
(line with an arrow, continued with periodic boundary conditions) are 312132, 311131, 1212, respectively.}
\label{fig:windows}
\end{figure}

\section{Main results:  syzygy sequences for the two center problem}

The main result is
\begin{thm} \label{thm:PerSyz}
The periodic syzygy sequences for regular periodic orbits of the two center problem are of the form:
\vskip .2cm
L type:  $1 3^{n_1} 2 3^{n_2} 1 3 ^{n_3} 2 \ldots 3^{n_s}$  (1's and 2's alternating)
\vskip .2cm
S type:  $1 3 ^{n_1} 1 3 ^{n_2} 1 3 ^{n_3} 1 \ldots 3^{n_s}\,\,\,$ or $\,\,\,2 3 ^{n_1} 2 3 ^{n_2} 2 3 ^{n_3} \ldots 23^{n_s}$ 
\vskip .2cm
P type:  $(12)^q$
\vskip .2cm
\noindent where the exponents $n_i$    are those of the Sturmian sequence \eqref{rotSeq} associated to the  rational rotation number $W$ for that orbit's torus 
(Definition~\ref{Sturmword}). The length of the fundamental sequence is $2(p+ q)$ where $W = p/q >0$.  
(For the range of possible rotation numbers see the final sentence of the  next theorem.)
\end{thm}

 See figure~\ref{fig:typeLorbits} for   examples of periodic type L orbits, their rotation numbers and  syzygy sequences
 for equal masses.  See  figure~\ref{fig:typeSorbits} for  examples of periodic type S orbits, their rotation numbers  and  syzygy sequences in a case of distinct masses. 
 
 \begin{figure}
\includegraphics[width=14cm]{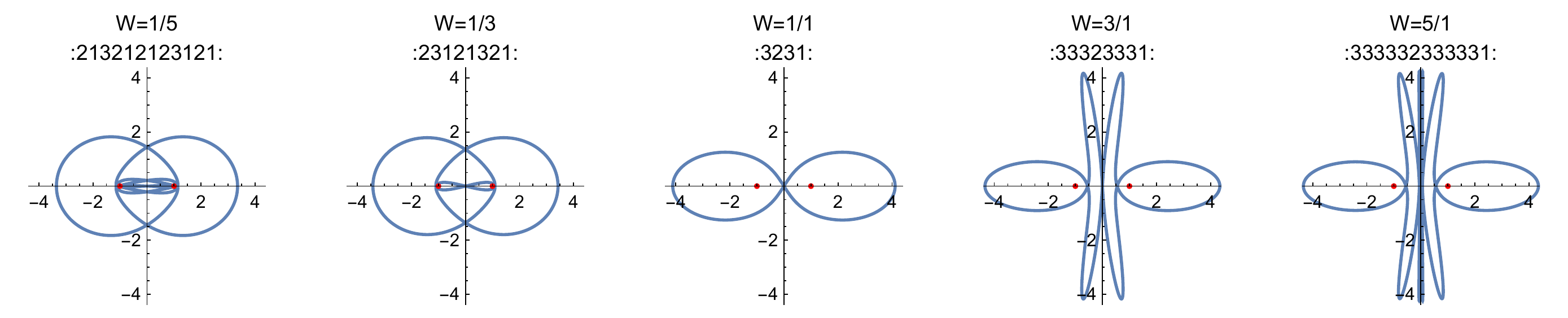}
\includegraphics[width=14cm]{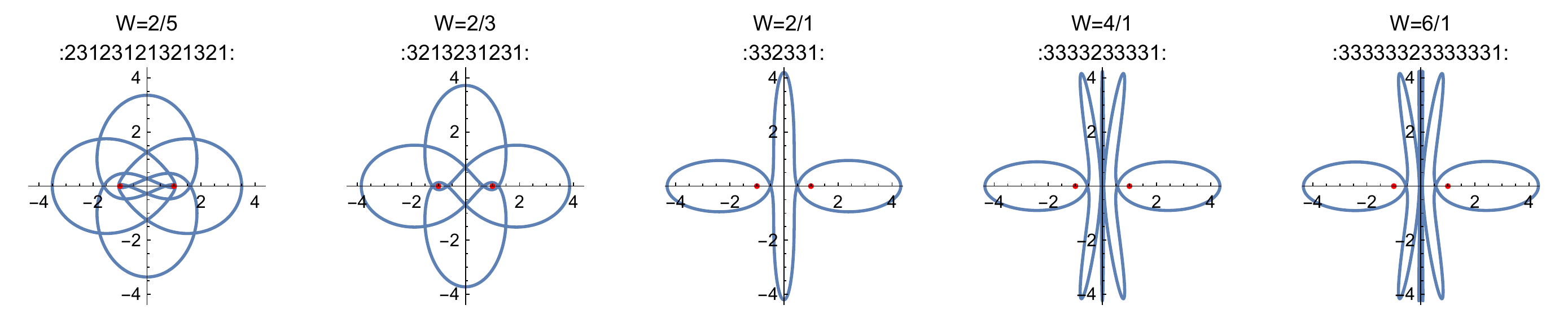}
\includegraphics[width=14cm]{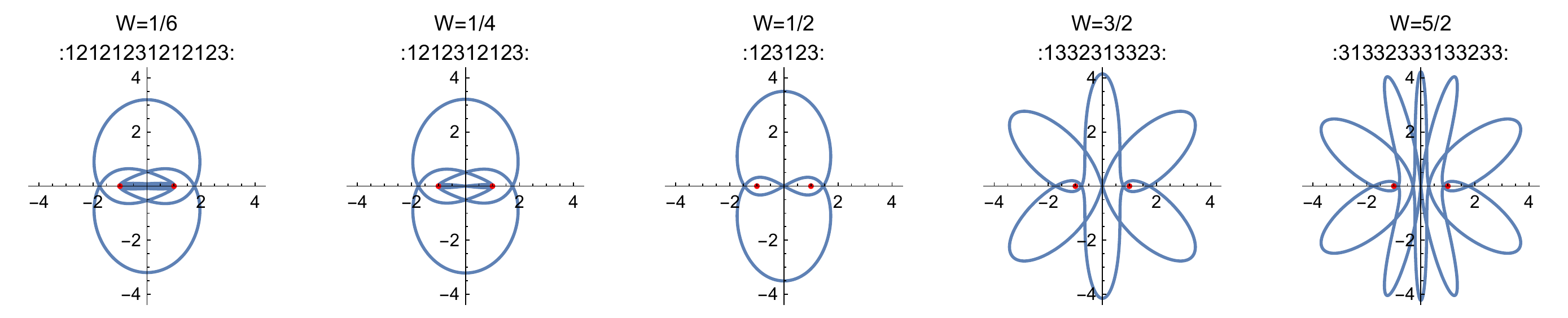}
\caption{Periodic orbits of Lemniscate type in the $(x,y)$ plane for the equal mass case $m_1 = m_2 = 1/2$, $d = 1$, $h = -0.23$
with various rotation numbers $W$. To the right more oscillations along the $y$-axis are added, to 
the left more oscillations along the $x$-axis are added, keeping the parity of $W$ the same along each row.}  \label{fig:typeLorbits}
\end{figure}

\begin{figure}
\includegraphics[width=8.4cm]{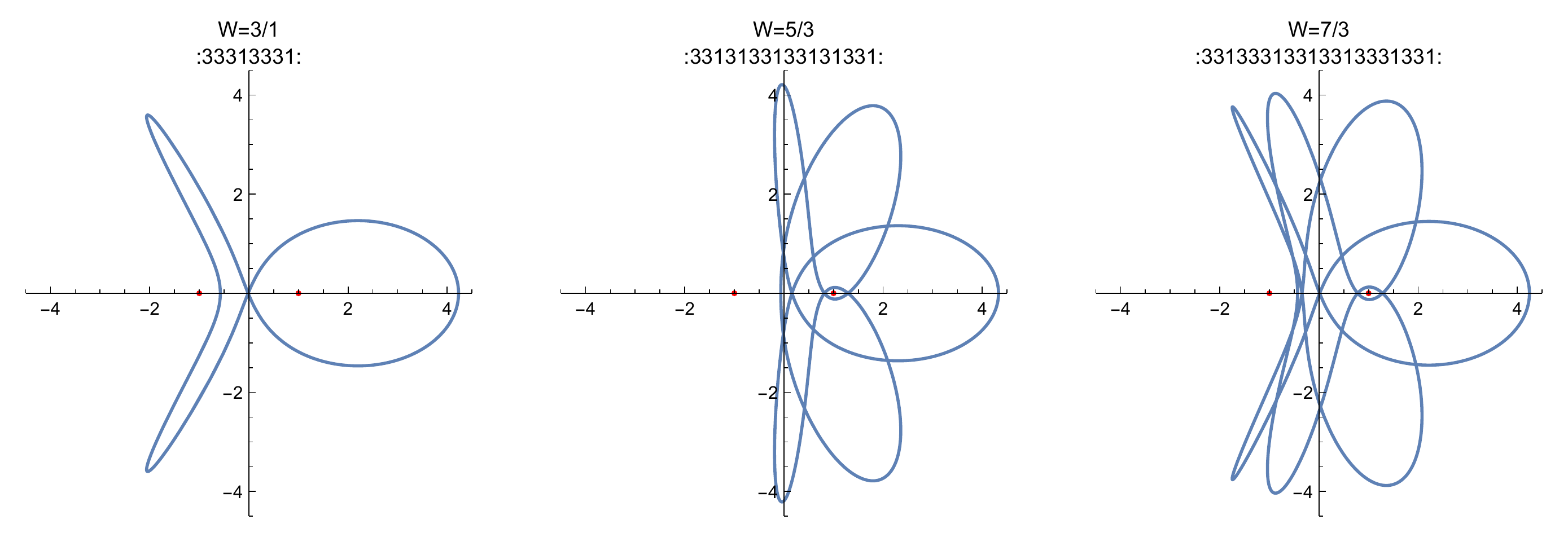}
\includegraphics[width=8.4cm]{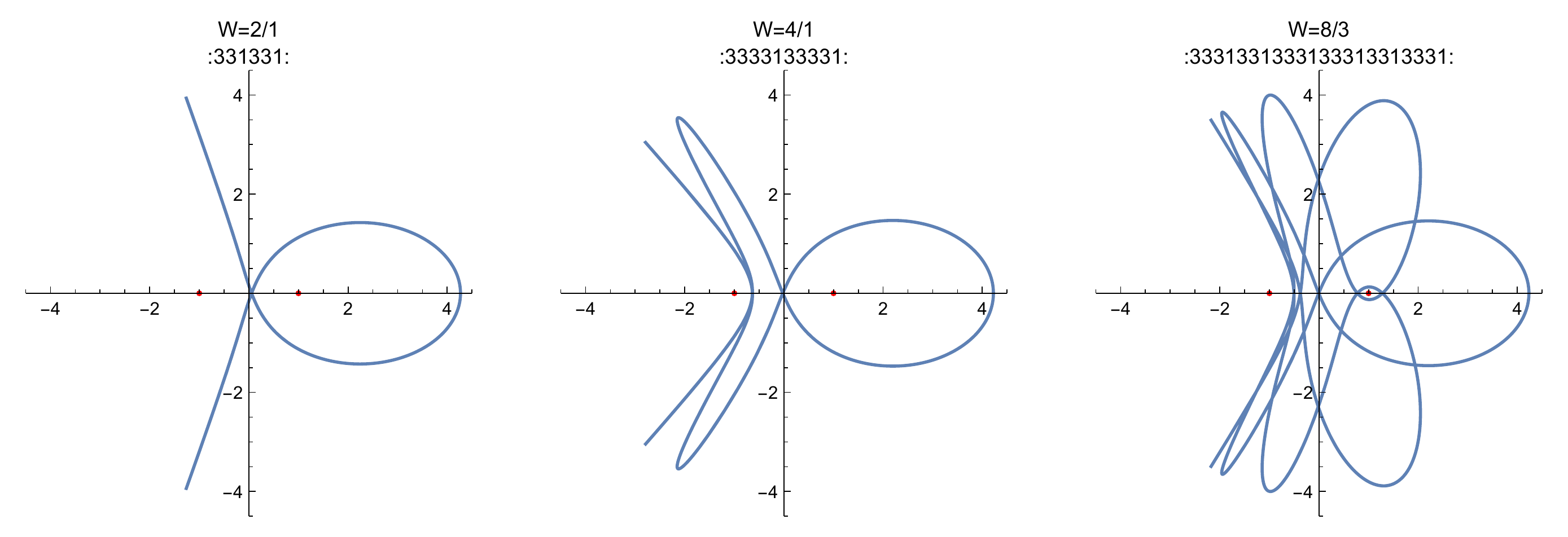}
\includegraphics[width=8.4cm]{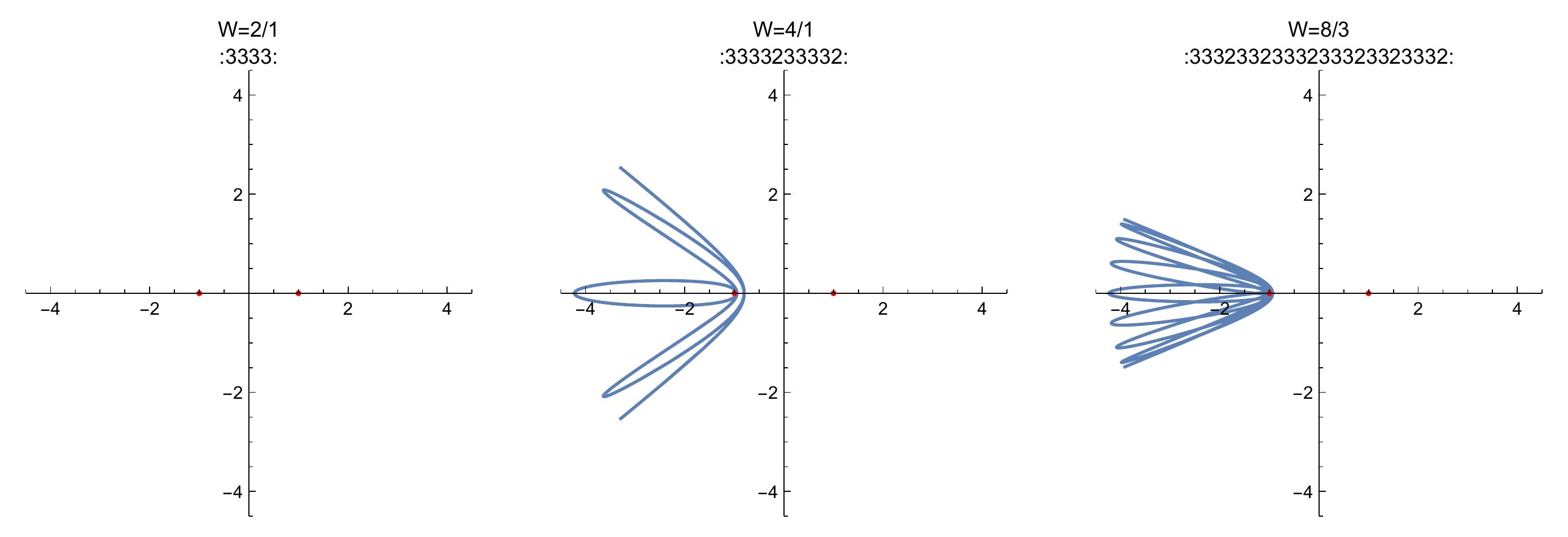}
\includegraphics[width=8.4cm]{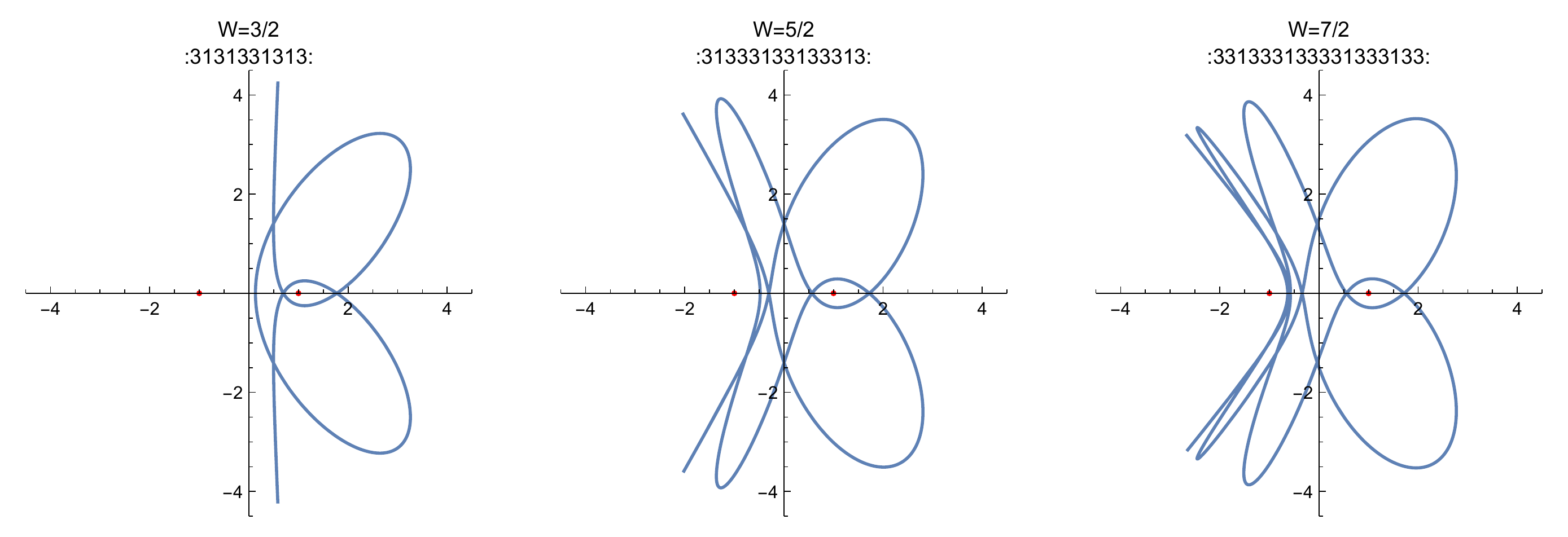}
\caption{Periodic orbits of Satellite type in the $(x,y)$ plane for the asymmetric case $m_1 = 1/3$, $m_2 = 2/3$, $d = 1$, $h = -0.23$
with various rotation numbers $W$. Parity if the same along each row. The third row shows tori of type S1, all other of type S2.
}  \label{fig:typeSorbits}
\end{figure}

 {\sc Missing periodic orbits.} The only periodic orbits whose sequences are not represented in the theorem above
 are those on non-regular level sets of the integral map. 
 There are five such families of periodic orbits corresponding to critical values of the integral map,
 see figure~\ref{fig:EM}. 
Three of these five correspond to periodic orbits that are collinear, i.e.\ they move along the axis 
$y=0$ that contains the two centers, and as such don't have well define syzygy sequences. 
The other two families have 
either syzygy sequence 3 (separating type L from type S, orbit is on the symmetry axis $x=0$ when then masses are equal), 
or syzygy sequence 12 repeated (orbit encircling both centers, the outer boundary of type P).

 Our main result is a corollary of the next theorem
 which  includes the infinite aperiodic orbits -- the dense windings on the tori.
\begin{thm} \label{thm:AllSyz}
The syzygy sequences   realized by  regular non-collision orbits for the two-center problem are precisely  one of  the following possibilities: 
$$\ldots 121212 \ldots, \text{in the  P case,}$$ 
$$
 \dots 1 3^{n_1} 2 3^{n_2} 13^{n_3} 2 \dots, \text{ in the L case,}
$$
$$
\dots 1 3^{n_1} 1 3^{n_2} 1 3^{n_3} 1 \dots, \text{ or } \dots 2 3^{n_1} 2 3^{n_2} 23^{n_3} 2\dots\text{ in the S case}
$$
In the last two cases  the integer sequence $n_i$ occurring  is the Sturmian sequence (Definition \ref{Sturmword}) associated to the rotation number $W = W(g,h)$ 
of that orbit's torus $\T(g,h)$,  and 
 any possible intercepts $y_0$.  
The  range of  rotation numbers $W$ determining the   $n_i$  can be read off from tables 
\ref{WtypeL} and \ref{WtypeS} below. 
\end{thm}

\begin{remark}  The rotation number  $W$ only depends on the regular values $g,h$ and not on the choice of torus
when $F^{-1} (g,h)$ consists of  more than one torus.  This ``coincidence'' is a consequence of the fact that the two-center problem
is integrable in elliptic functions, and moreover that the differentials that define the rotation numbers are of first kind.
\end{remark}



The main theorem has some simple consequences.

\begin{cor}
In any orbit on a regular torus of type L or S the number of consecutive 
3s is either $n$ or $n+1$ where $n$ is the integer part of $W$.
\end{cor}
\begin{proof}
This follows directly from the definition of the Sturmian sequence.
\end{proof}

\begin{cor}
In all cases symbol 1 and   2 never appear adjacent to each other (``no  1 or 2 stutters'').
Lemniscate orbits with $W < 1$ have no stutters, while for $W > 1$ they have 3-stutters.
Every Satellite orbit has at least one 3-stutter.
\end{cor}
\begin{proof}
For type L the symbols 1 and 2 alternate, so there cannot be stutters of these symbols.
For type S there would be 1- or 2-stutters if any $n_i =0$, however we have numerically verified that 
$W > 1$ for type S, so by equation \eqref{rotSeq} we have $n_i >0$.
At least one 3-stutter, i.e.\ one $n_i=2$, must occur when $W > 1$, because then the line in figure~\ref{fig:windows}
with slope $>1$ must intersect the vertical lines of symbol 3 at least once. 
\end{proof}

\section{Guts of the proof.} 
\label{guts} 

\begin{figure}
\includegraphics[width=7cm]{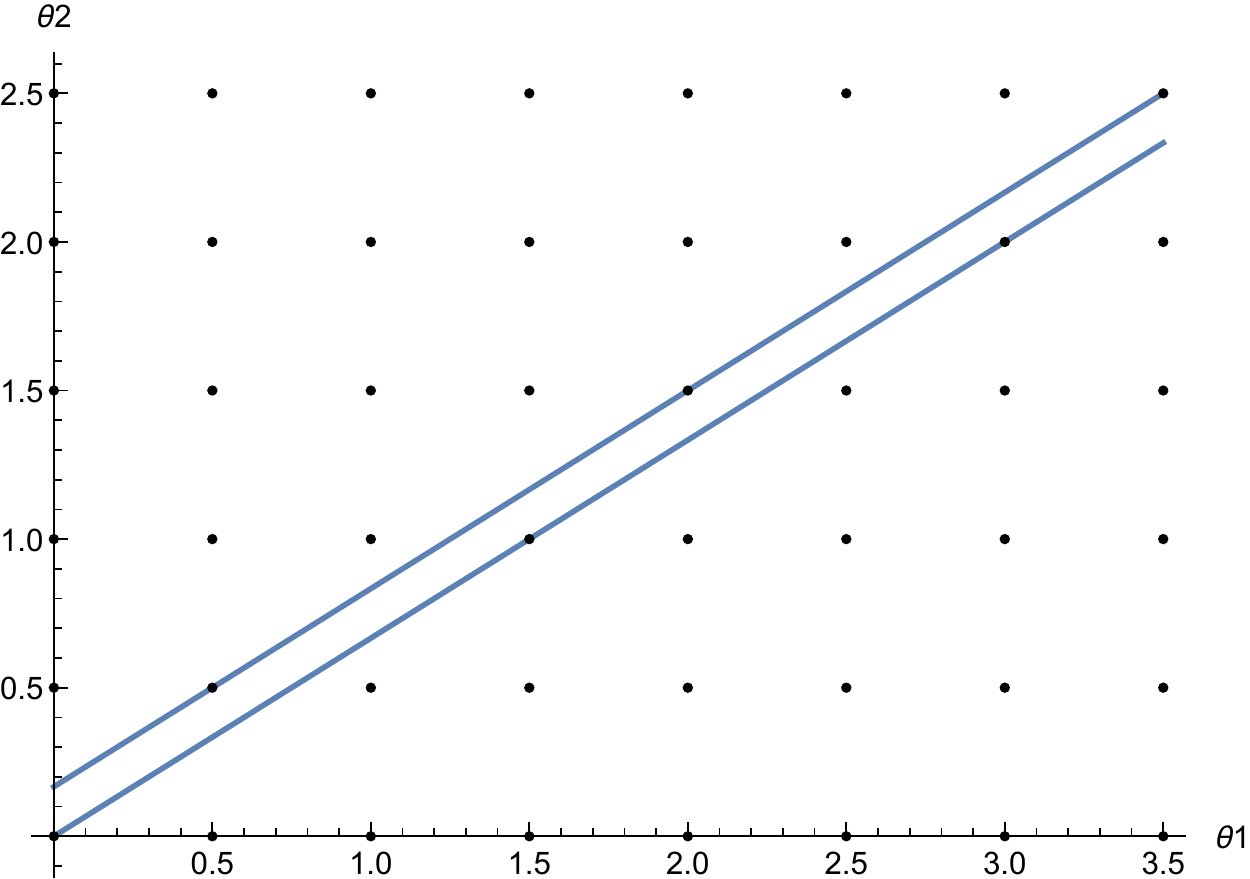}
\caption{Two solution curves with $W = 2/3$ bounding a region of orbits without collisions. 
Half-lattice points correspond to collisions.} \label{fig:perla}
\end{figure}

Each regularized  Liouville torus $\T = \T(g,h)$ maps to configuration space by composing its inclusion  into
regularized $\lambda, \nu, p_{\lambda}, p_{\nu}$-phase space and   projecting the regularized phase space onto the usual
$xy$ configuration space.   The windows on $\T$ are the inverse image of the collinear locus (the 
$x$-axis) with respect to this map.  We will show in sections \ref{sec:Euler} and \ref{halfaunit} that these windows are a   finite  collection of horizontal and vertical lines,
positioned as indicated in figure~\ref{fig:windows} according to the cases S, L or P.
Here ``horizontal'' or ``vertical'' mean  that when we coordinatize $\T$  in standard angle variables $\theta_1 \theta_2$  of ``action-angle'' 
as  $\R^2/\Z^2$, then these lines are parallel to the $\theta_1$ or the $\theta_2$ axes.
Lifted to the universal cover, then we get almost exactly the `graph paper picture' used for writing Sturmian words. 
In the S and L  cases we have 4 lines in all, a horizontal pair and a vertical pair, 
see figure~\ref{fig:windows}.  
The parallel lines in each pair are separated by 1/2 a lattice unit.  In the P case we have only the two vertical lines
separated by 1/2 a unit.  
The intersections of the lines represent collisions of our test particle with the appropriate primary.
This pattern is represented on the plane as a doubly periodic pattern of parallel lines.

On  a fixed torus $\T = \T(g,h)$ the solution curves are a family of parallel  straight lines $y = W x + b$  with $W = W (\T)$ constant,
and $b$ varying.  Here we have changed variables so $x = \theta_1$, $y = \theta_2$
so as to conform to the discusion  of subsection \ref{Sturm}.  The torus supports periodic solutions if and only if $W$ is rational.  
A finite number of these solutions will have collisions.  
Those remaining have a syzygy sequence.

 \begin{figure}
 \includegraphics[width=12cm]{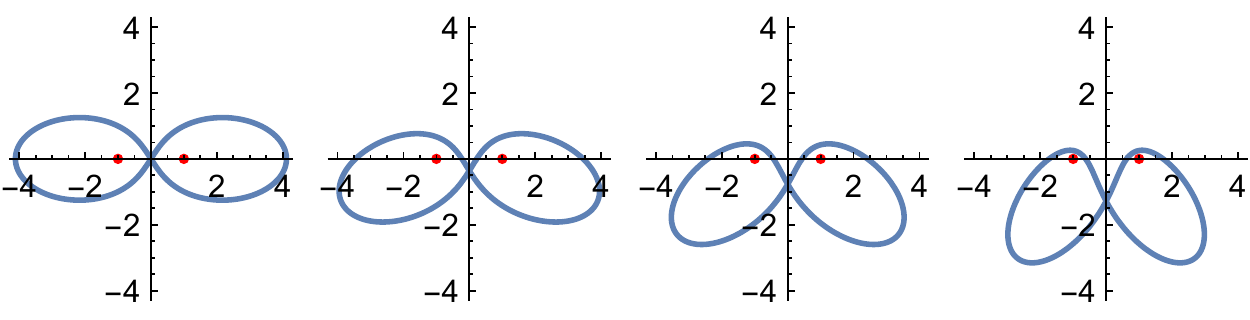}
\includegraphics[width=12cm]{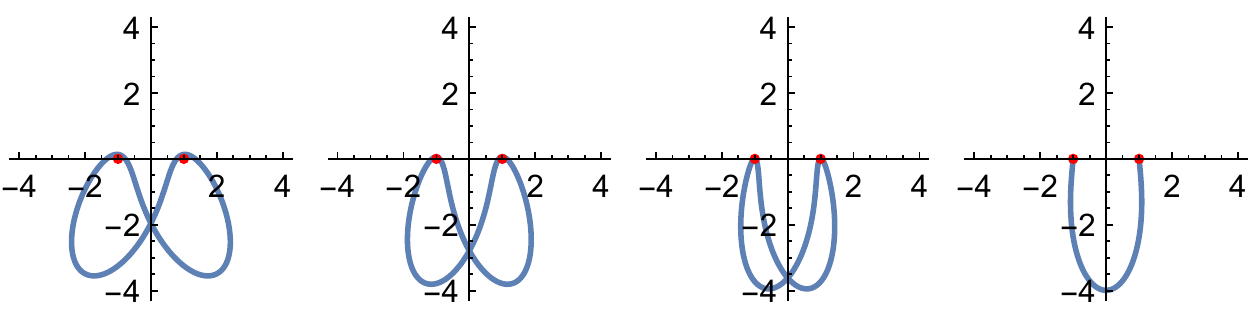}
 \caption{Orbits of type L with $W = 1$ and syzygy sequence $1323$ for $h=-0.23$ and $m_1 = m_2 = 1/2$ for 
 a sequence of initial conditions with different $\nu_0 = i \pi /14 $, $i=0, \dots, 7$ 
 and the same $\lambda_0 = 0$  on the same torus.  For $\nu_0 = 0$ the orbit 
 has additional discrete symmetry, while for $\nu_0 = \pi/2$ it is a collision-collision orbit.
 } \label{fig:W11icodeform}
 \end{figure}

The following is essentially a restatement of theorem 2: 
\begin{lem}  \label{lem:combi}
On a fixed regular rational torus $\T =\T(g,h)$ all collision-free solutions have the same periodic syzygy sequence. 
This sequence is as described in theorem~\ref{thm:AllSyz}, where $W = W(g,h)$ is the rotation number of $\T$. 
If   $W = p/q$ then the length of this common syzygy  sequence is  $2(p+q)$ in the S and L cases, and $2q$ in the $P$ case. 
\end{lem}

\begin{proof} 
Let $W = p/q$ be given, and the labelled set of  vertical and (possibly) horizontal lines be given.
Each solution is a straight line of slope $p/q$, see figure~\ref{fig:perla}

\textit{P case.}  There are no horizontal lines,  so no collisions.  Vertical  lines  marked with 1 and 2 alternate. 
The only possible sequence  is  $(12)^s$ for some power $s$.  
Each straight  line cuts across $q$ basic units in the horizontal direction before closing, so that we have $s = q$.  

\textit{S and L cases.}
Fix attention on one of the 4 collision points: these being the intersection of the vertical and horizontal lines.
Shift the fundamental domain describing the torus so that this point is at the origin.  Collision solutions are those
passing through a lattice point and we  throw these out.  Because the line pairs 
are half a unit a part, it makes sense to magnify the lattice. Consider the new lattice  $((1/2) \Z) ^2$
in the plane,   the torus's universal cover, whose points represent collisions.    We are  now in a situation identical to that 
of subsection \ref{Sturm}, except that the lines have different labellings!  The independence of the sequence on
the initial condition of the torus follows from lemma~\ref{lem:rational} of subsection \ref{Sturm}. Keeping track of the labelling 
leads to the sequences described in theorem 2. 
Everything has already been  done by  Morse and Hedlund, as described in section \ref{Sturm}, with the exception of the counting of sequence length. 
For this count, observe that   there are 2 horizontal lines per vertical  unit
and   2 vertical lines per horizontal unit.  Every  solution must traverse $q$ horizontal units
and $p$ vertical units and so crosses $2q + 2p$ lines in all.

\end{proof}
  
\begin{remark} The fact that the symbol sequences are the same for solutions on a fixed rational torus does  not  mean that these solutions are the same!  
Indeed, they sweep out a torus-full of solutions. 
In figure~\ref{fig:W11icodeform} we show the deformation of the type L orbit with $W = 1/1$ when the 
initial condition on the torus is changed. The orbit undergoes a deformation from a symmetric state
all the way to a collision orbit.  
Note that for a generic non-integrable perturbation one expects that almost all of these periodic 
orbits will be  destroyed. 
 \end{remark}
 
\section{Collision orbits}

 \begin{figure}
\includegraphics[width=14cm]{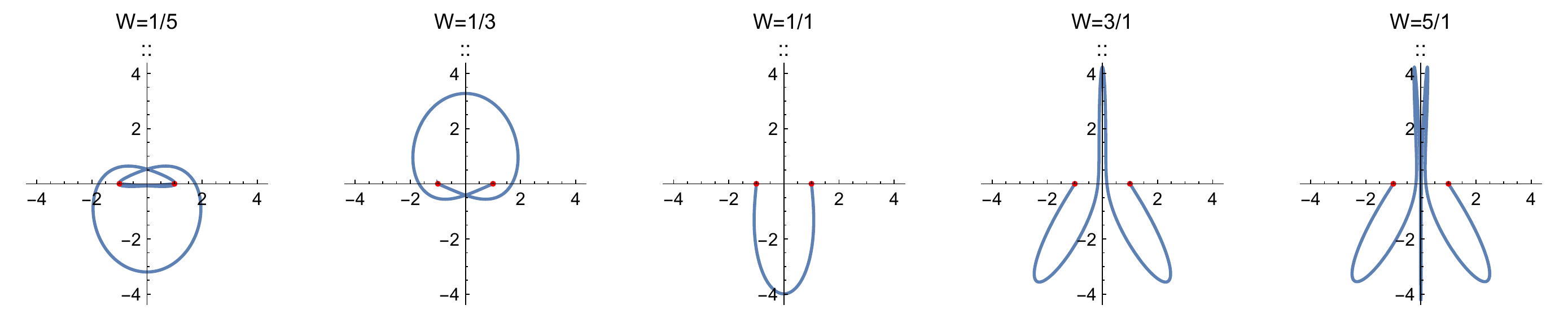}
\includegraphics[width=14cm]{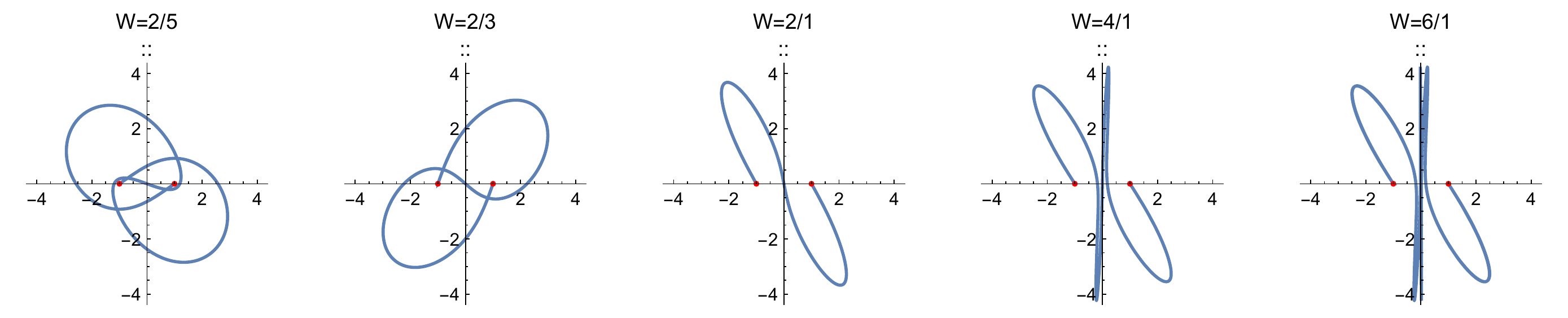}
\includegraphics[width=14cm]{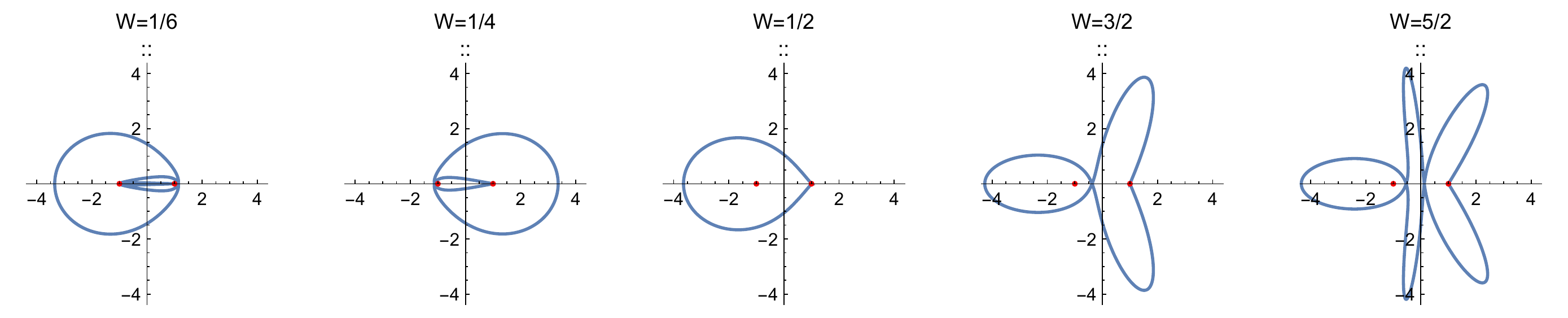}
\includegraphics[width=8.4cm]{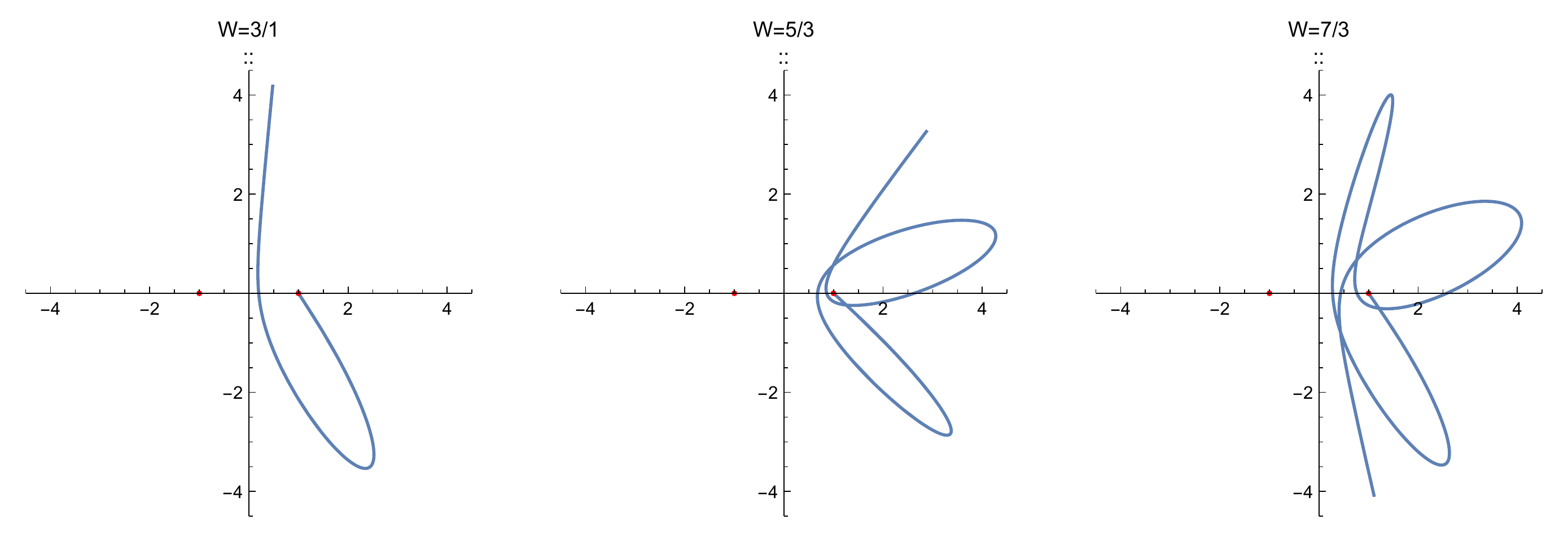}
\caption{Collision orbits of Lemniscate type in the $(x,y)$ plane for the equal mass case $m_1 = m_2 = 1/2$, $d = 1$, $h = -0.23$
with various rotation numbers $W$. To the right more oscillations along the $y$-axis are added, to 
the left more oscillations along the $x$-axis are added, keeping the parity of $W$ the same along each row.
Collision orbits with $W = odd/even$ visit the same collision twice.
The last row shows some brake-collision orbits of Satellite type (same parameters).}  \label{fig:CollisionOrbits}
\end{figure}

In figure~\ref{fig:CollisionOrbits} we show some collision orbits, which are found on tori with 
rational $W$ for initial conditions with $\lambda = 0$ and $\nu = \pm \pi/2$, as we now show.

\begin{cor}
Orbits that connect two collisions lie on a regular torus with rational rotation number $W$.
\end{cor}
\begin{proof}
Collision points are half integer lattice points in the covering plane of the torus. 
Any line connecting such lattice points has rational slope.
\end{proof}

\begin{cor}
Orbits that have a single collision lie on a regular torus with irrational rotation number $W$.
\end{cor}
\begin{proof}
A line with rational slope that hits a half integer lattice collision point in forward time,
will also hit another such point in backward time. Hence only lines with irrational slope 
can hit one lattice point (in either forward or backward time).
\end{proof}

\begin{cor}
Regular orbits that keep a finite distance from collisions lie on a regular torus with rational rotation number $W$ 
(unless in the P family, which has no collisions).
\end{cor}
\begin{proof}
A line with irrational slope will come arbitrarily close to a half integer lattice point. 
Hence the only lines that keep a finite distance have rational slope, 
unless they are collision-collision orbits, as described in the previous corollary.
\end{proof}

\begin{cor}
Any regular torus of type S or type L with rational rotation number has exactly two collision-collision orbits.
\end{cor}
\begin{proof}
This follows from lemma~\ref{lem:combi}, the fundamental region covered by 
two copies of the strip shown in figure~\ref{fig:perla} has exactly two collision-collision
orbits, corresponding to the two lines shown in the figure.
\end{proof}

\section{Euler's Problem of Two Fixed Centers}
\label{sec:Euler}

\begin{figure}
\centering{ 
\includegraphics[width=4.5cm]{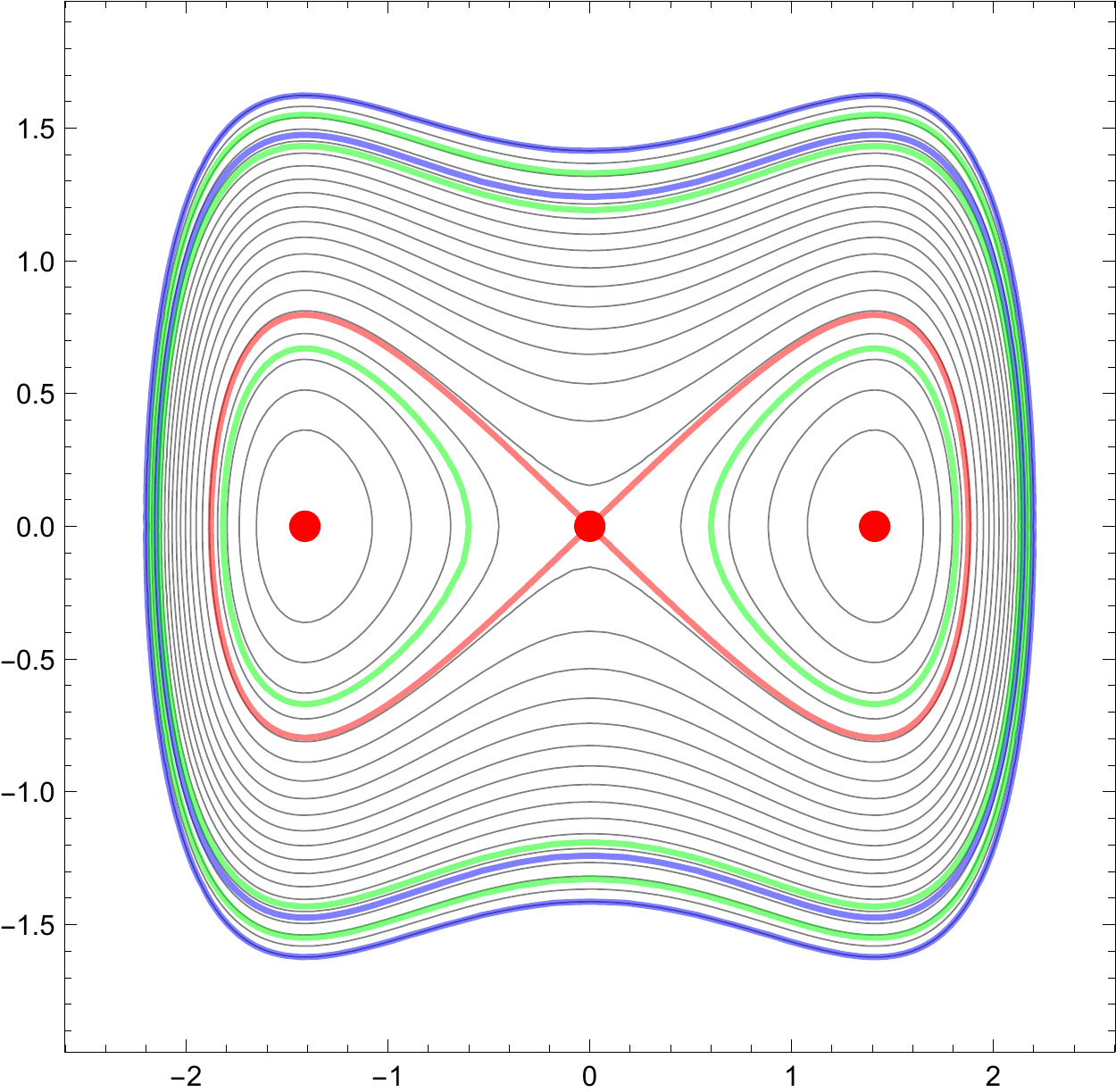}
\includegraphics[width=4.5cm]{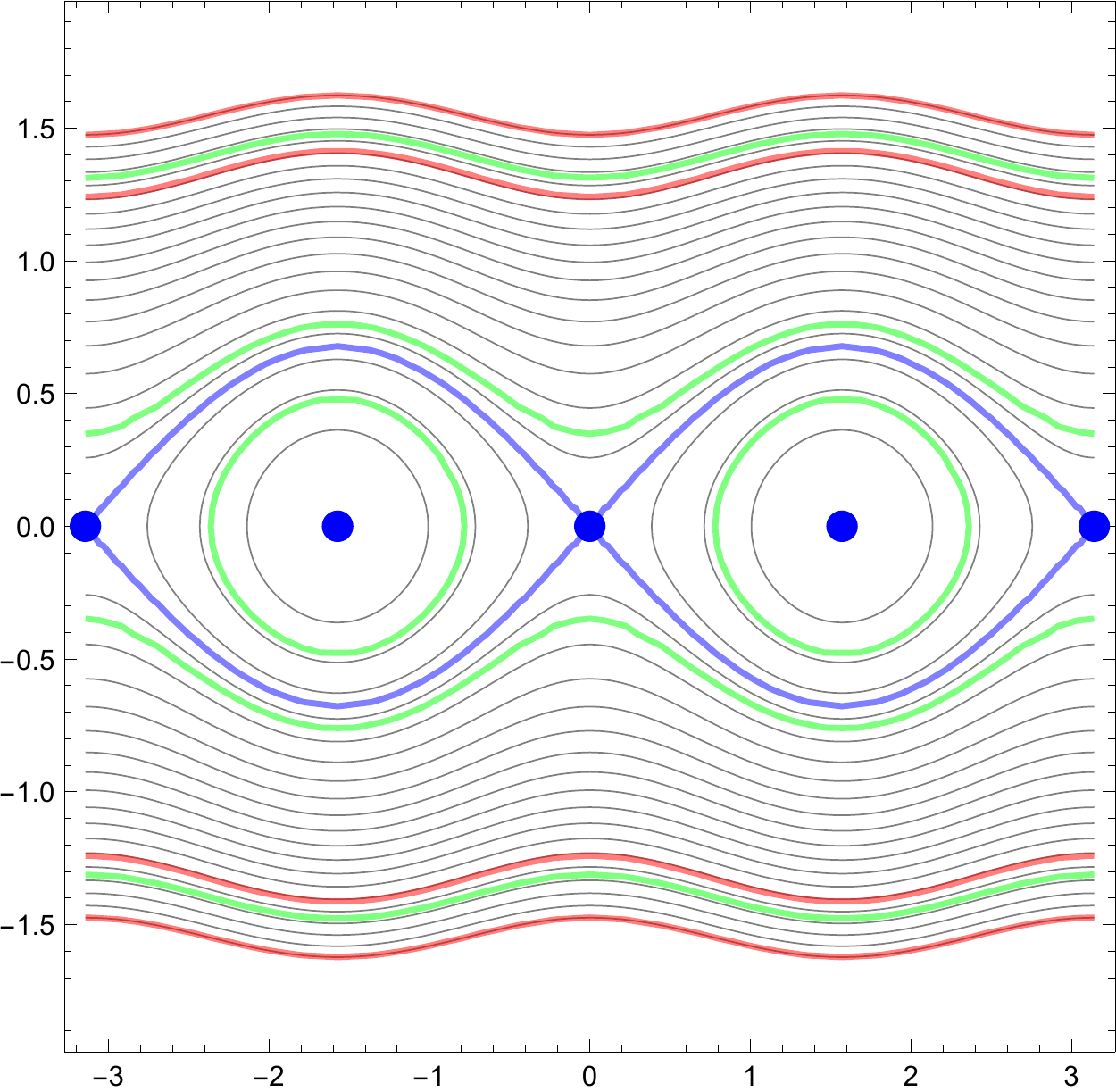}
} \\
\centering{ 
\includegraphics[width=4.5cm]{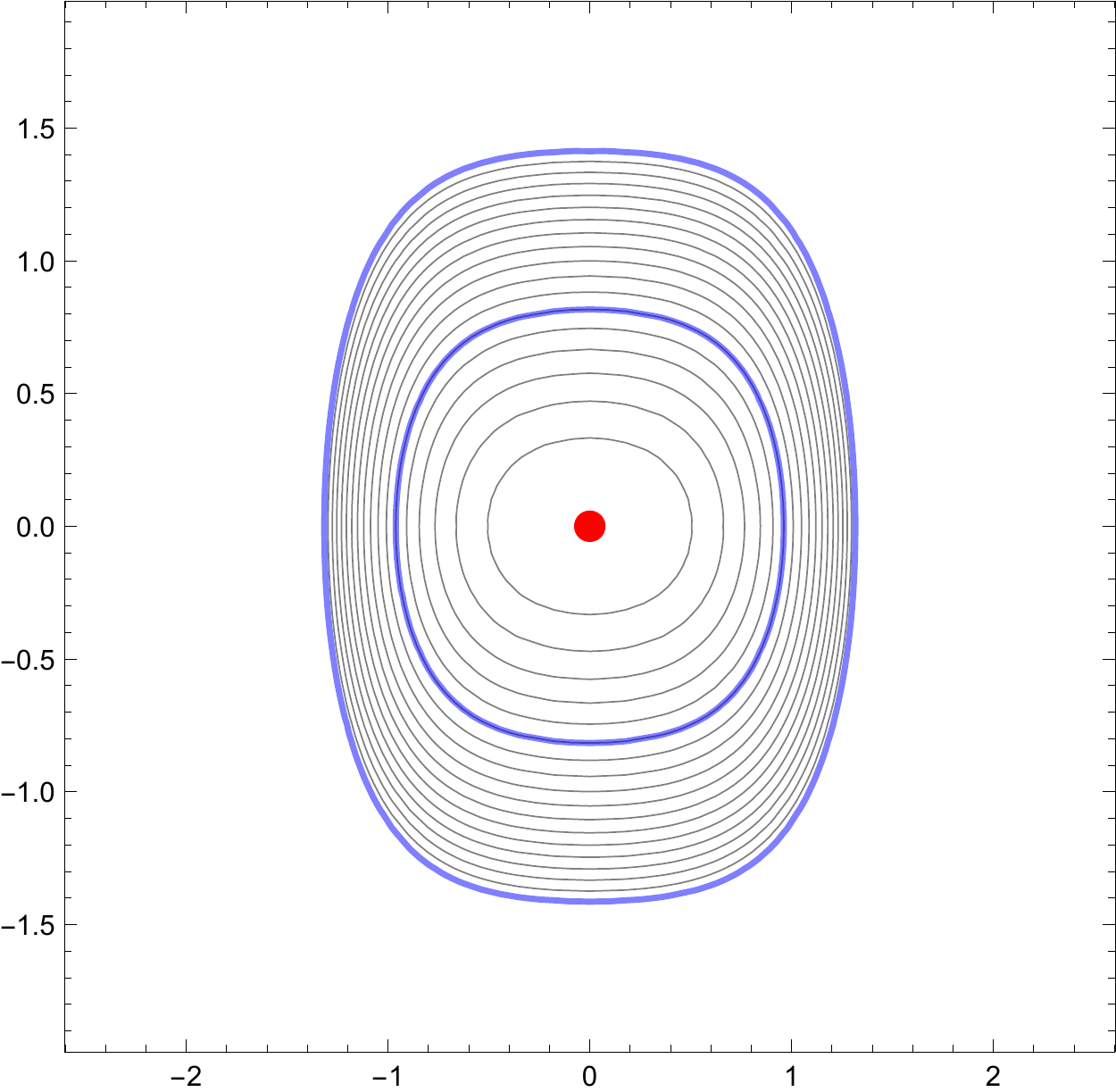}
\includegraphics[width=4.5cm]{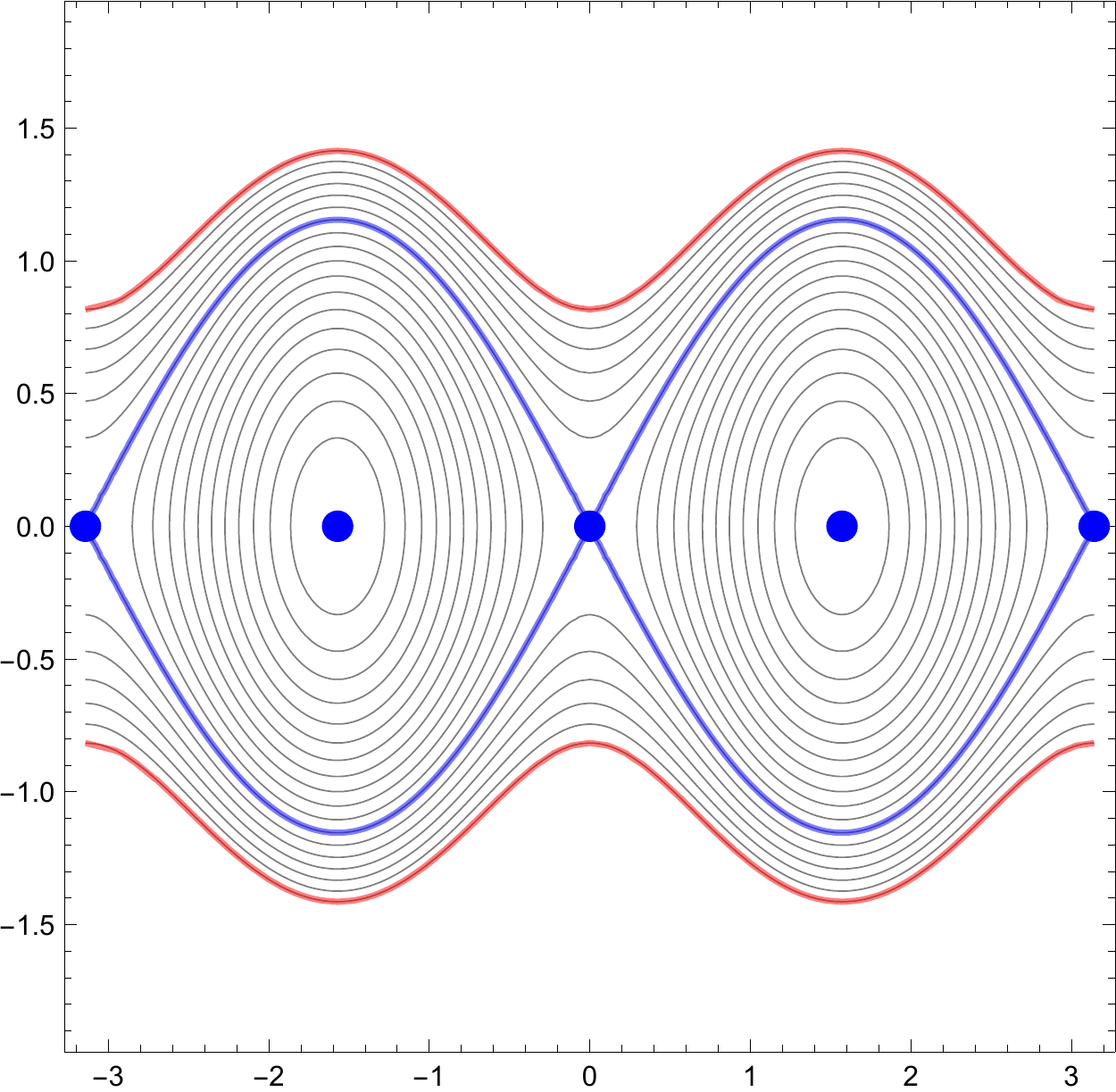}
} \\
\centering{ 
\includegraphics[width=4.5cm]{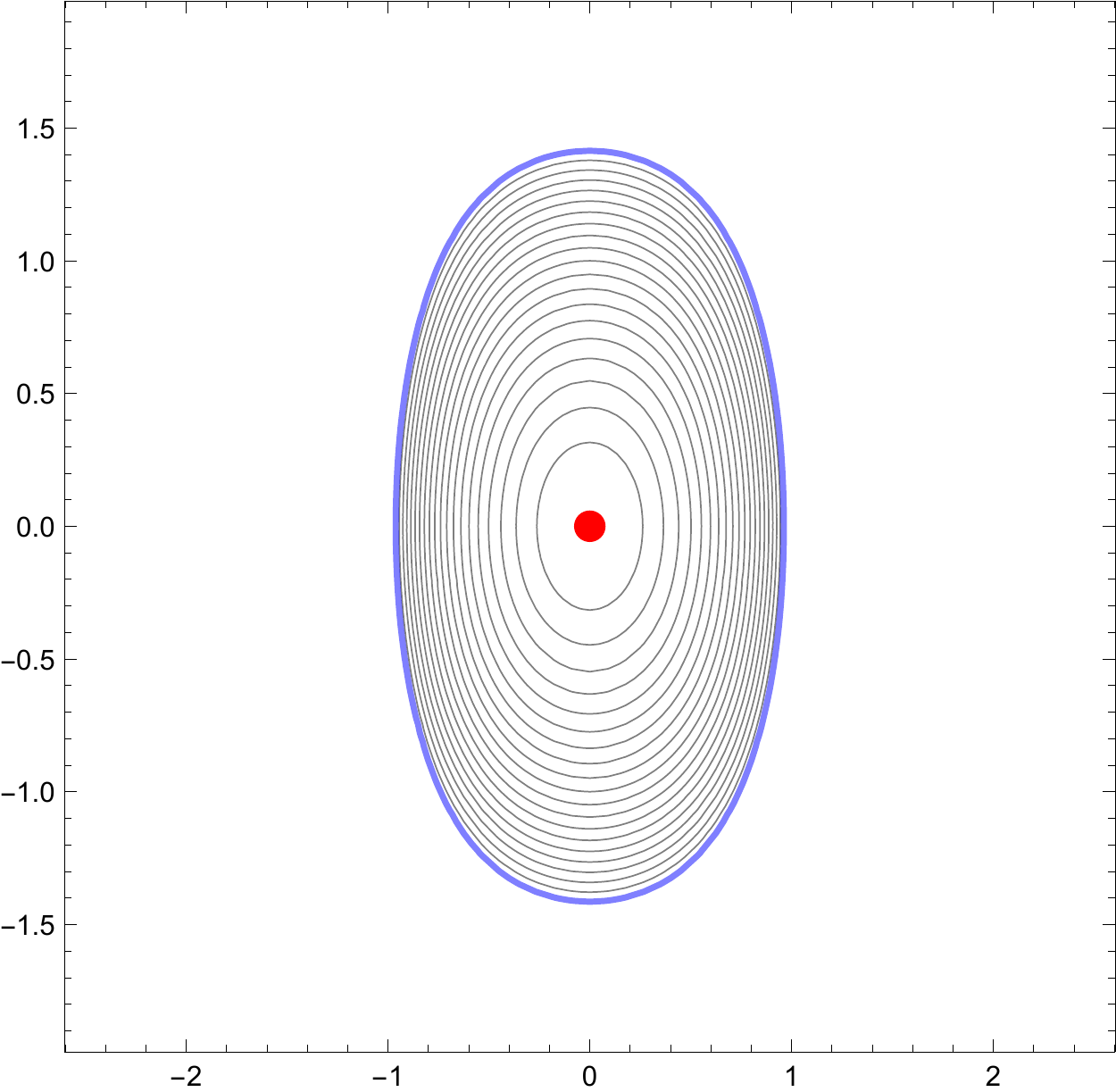}
\includegraphics[width=4.5cm]{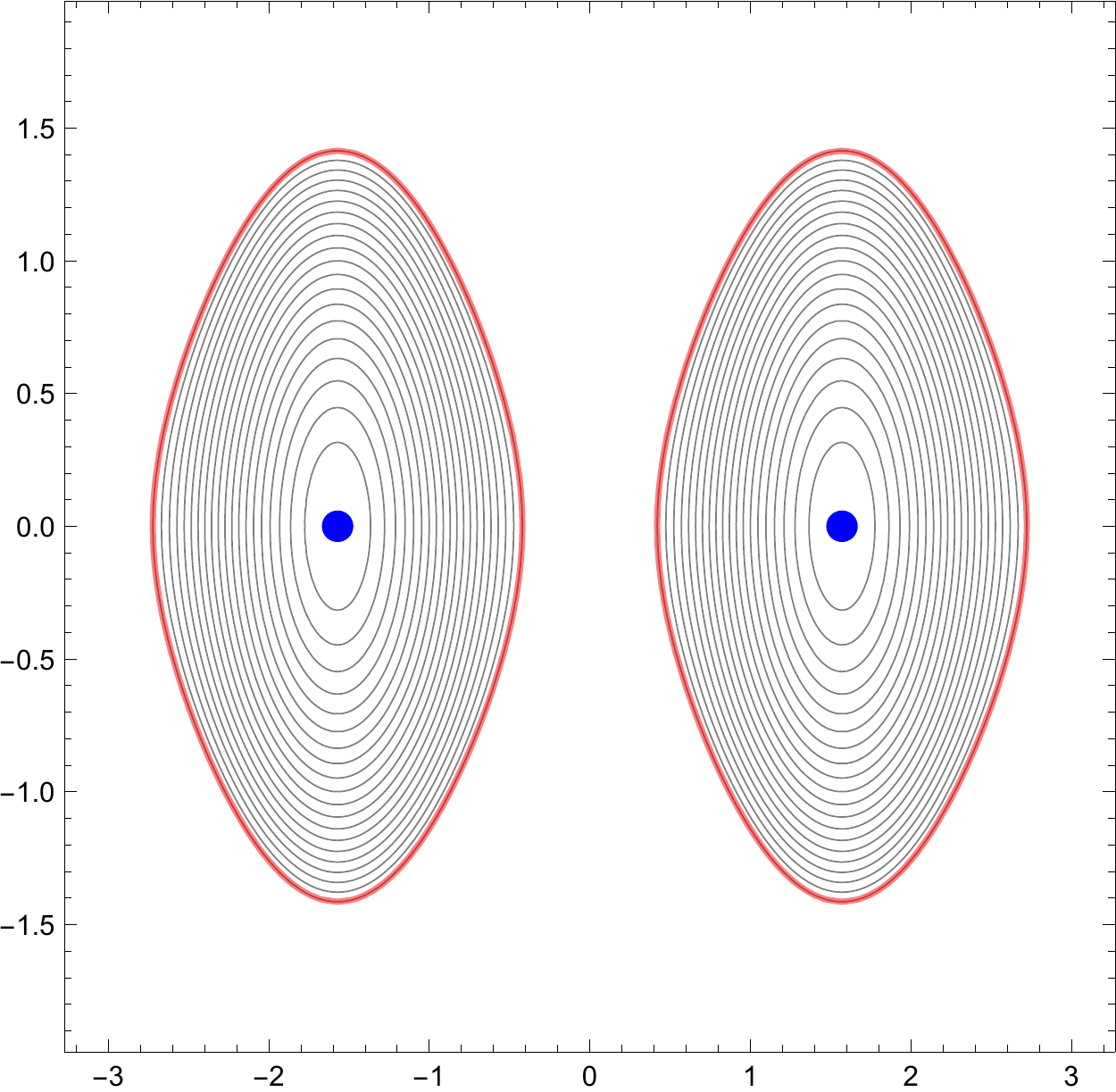}
}
\caption{Level lines of $H_\lambda$ (left) and $H_\nu$ (right)
with critical points marked by red and blue dots, respectively, 
and pre-images of critical values in the corresponding color. 
Green contours denote the levels corresponding 
to the three example shown in Fig.~1.
Always $m_1 = m_2 = 1/2$, $d = 1$. 
Top: $h = -1/6 > h_\lambda $,  types S, L, and P occur. 
Middle: $h = -2/3 > h^* $,  types S, and L occur. 
Bottom: $h = -6/5 < h^*  $,  type S only.}
\label{fig:CPsym}
\end{figure}

\begin{figure}
\centering{ 
\includegraphics[width=4.5cm]{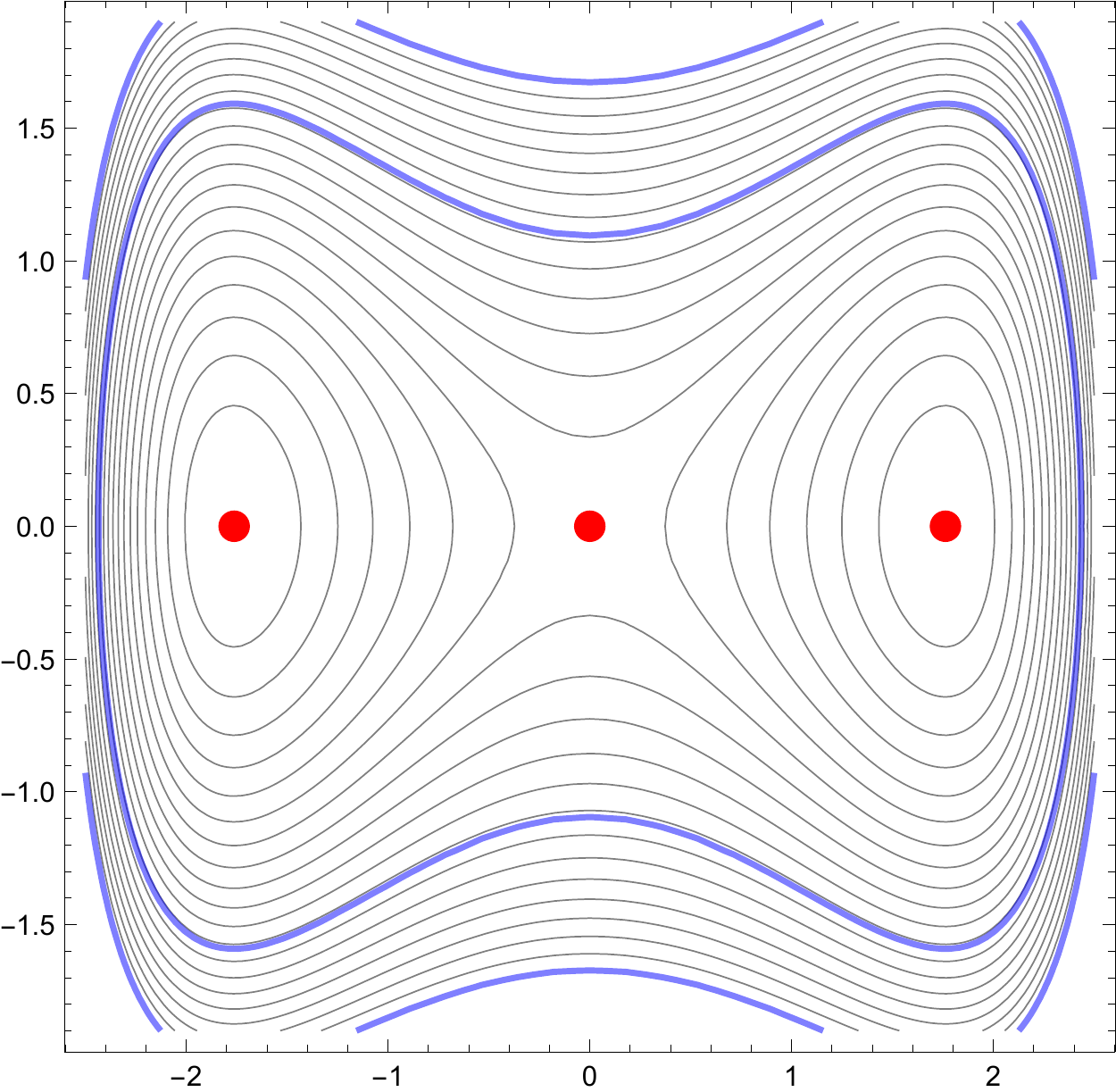}
\includegraphics[width=4.5cm]{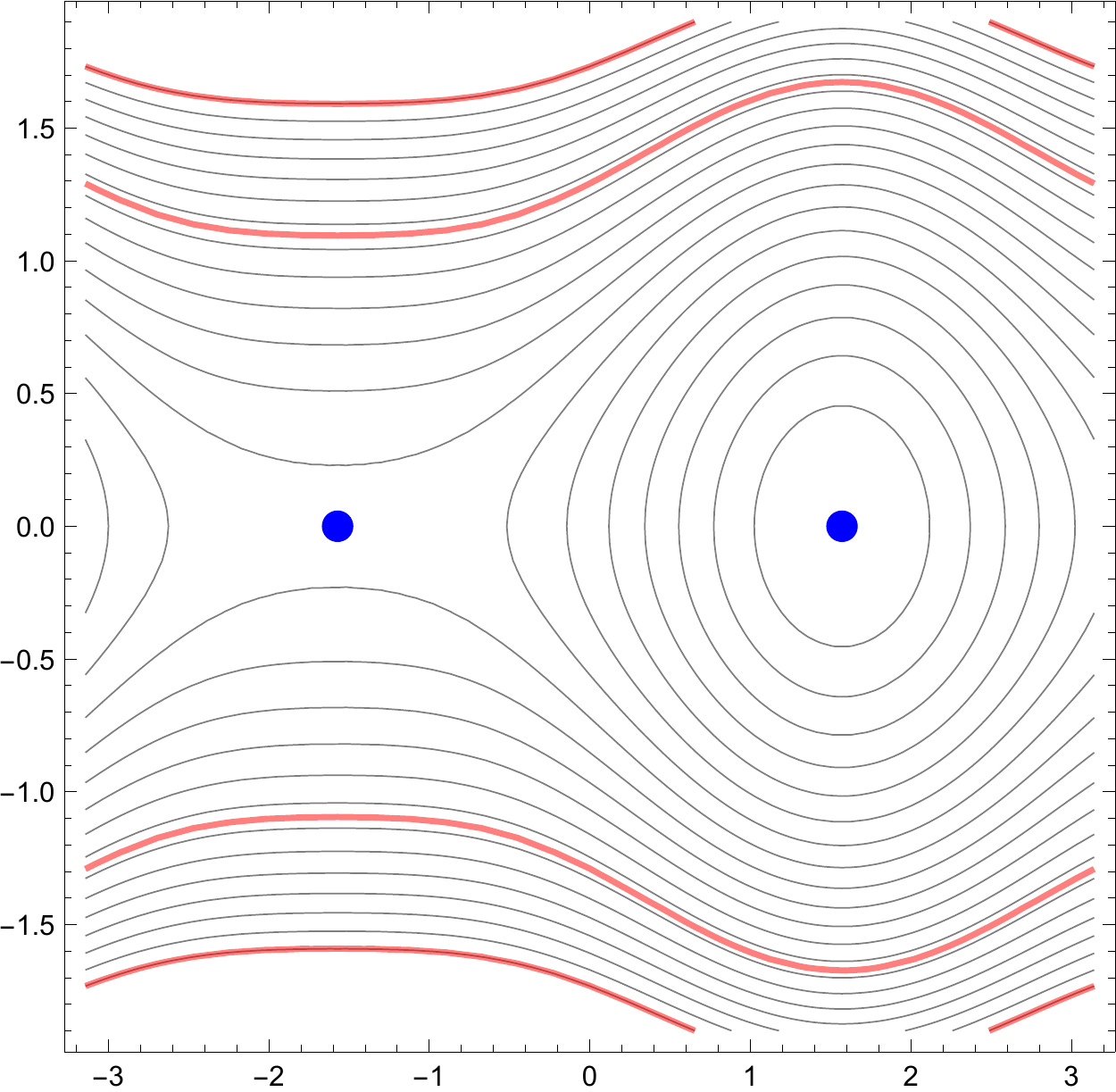}
}\\
\centering{ 
\includegraphics[width=4.5cm]{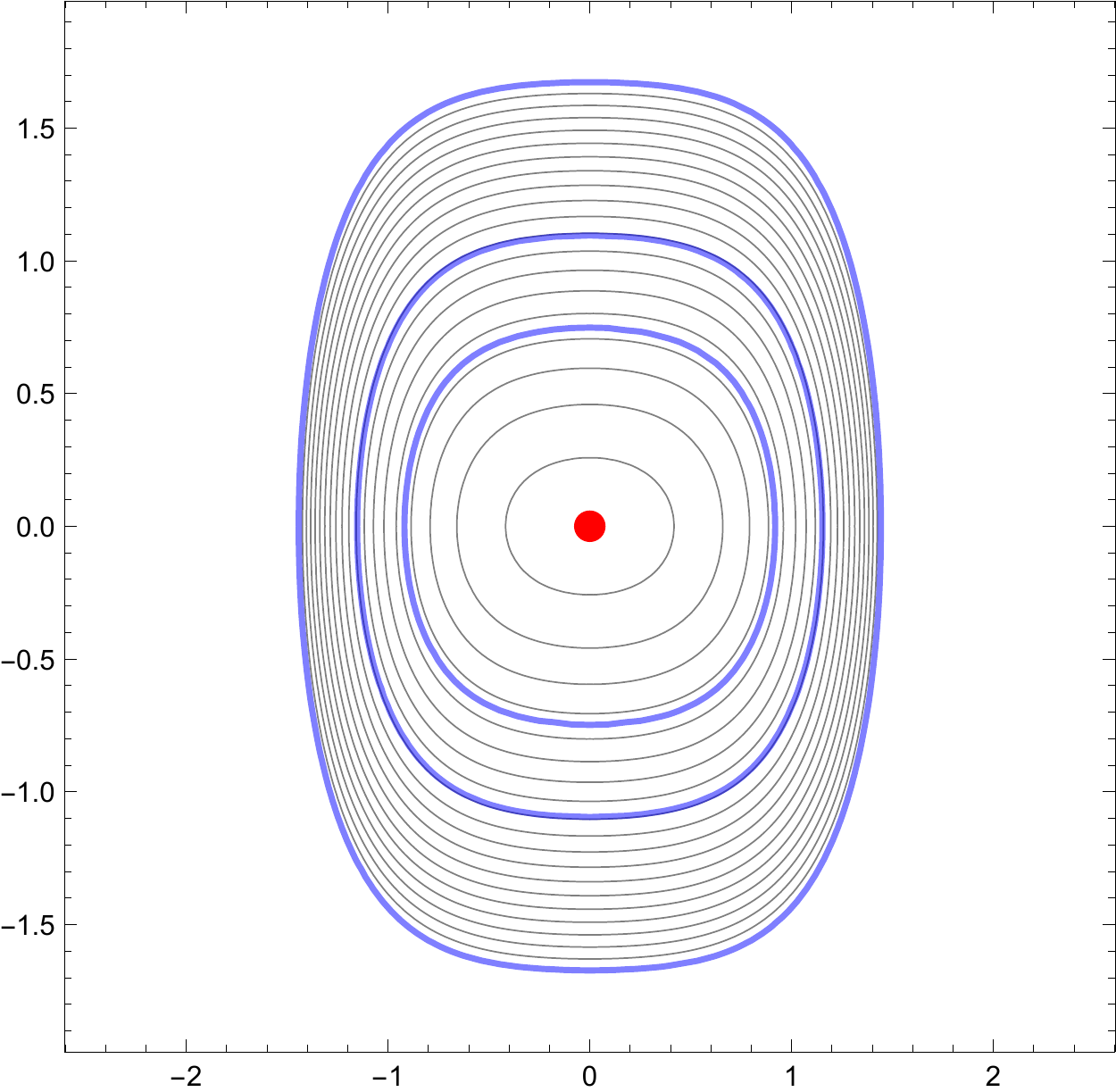}
\includegraphics[width=4.5cm]{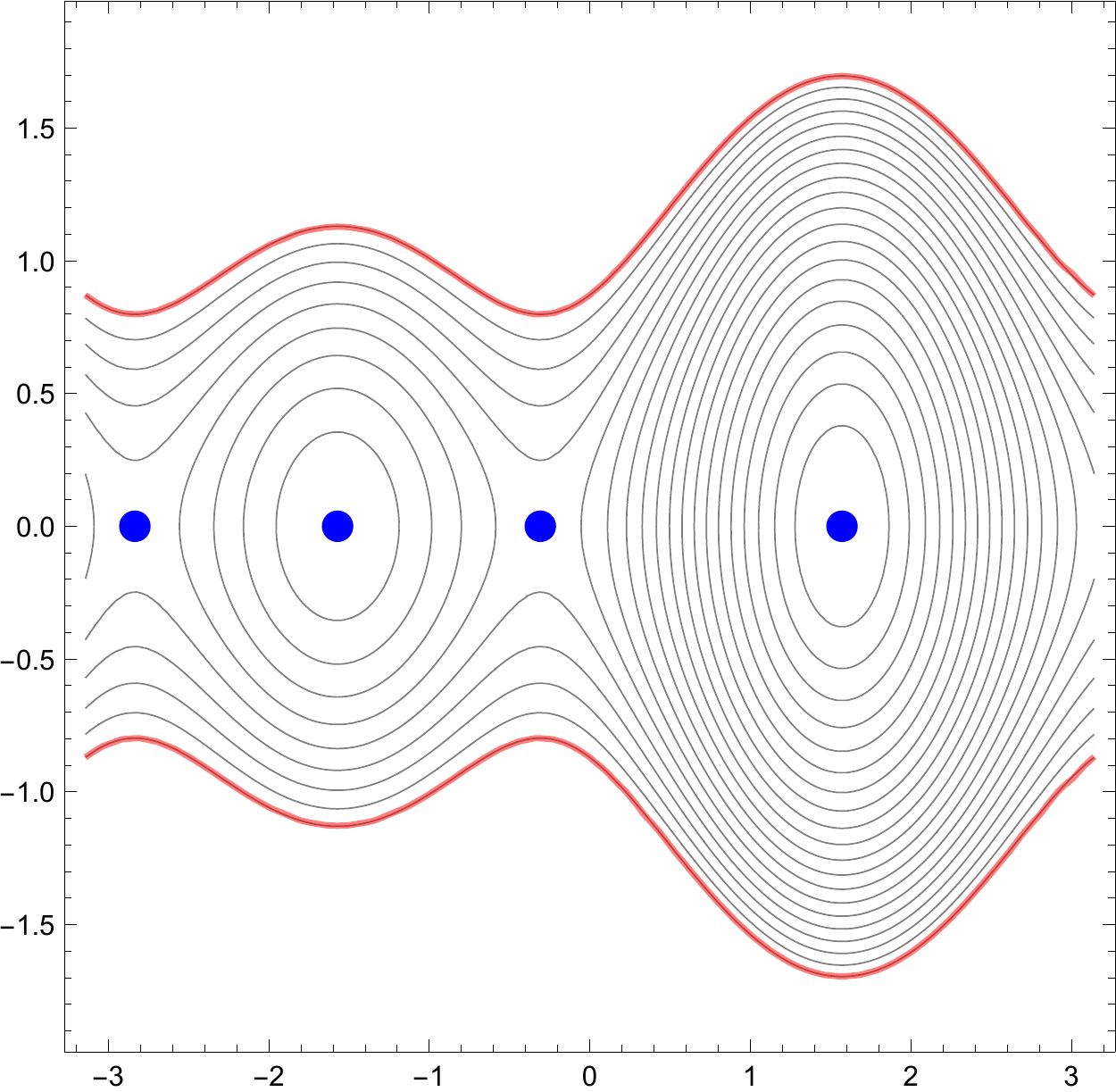}
}\\
\centering{ 
\includegraphics[width=4.5cm]{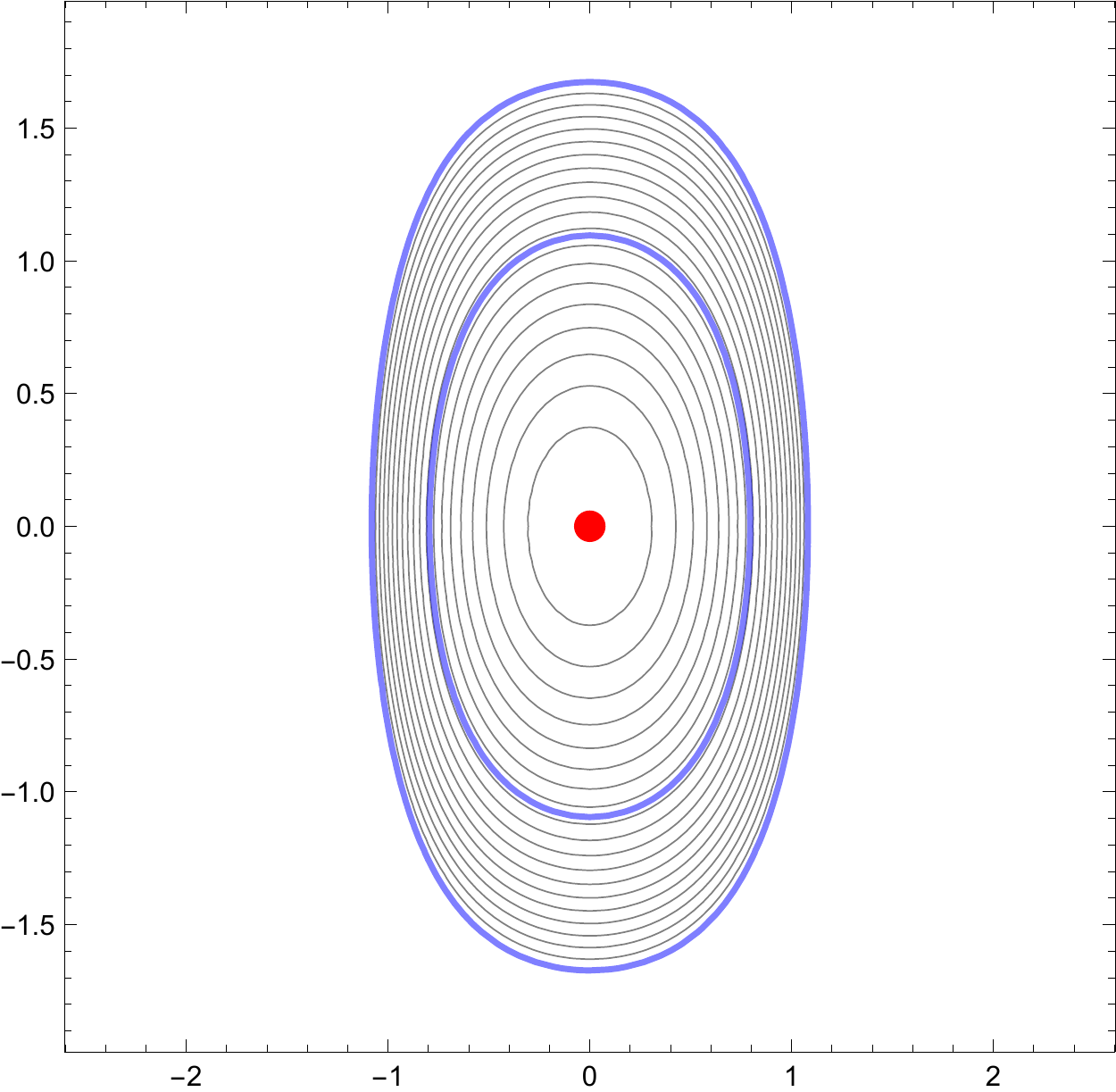}
\includegraphics[width=4.5cm]{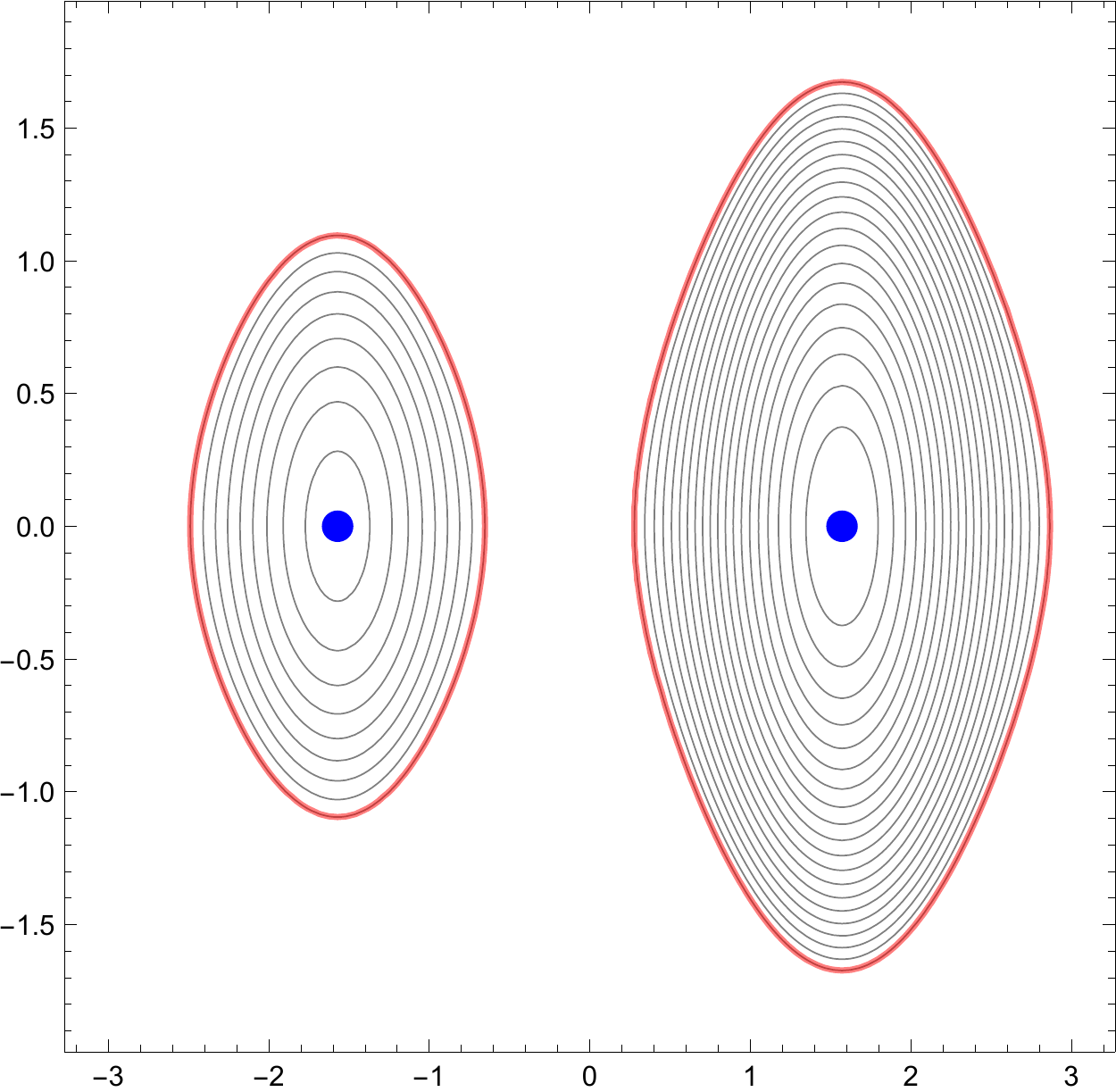}
}
\caption{Level lines of $H_\lambda$ (left) and $H_\nu$ (right)
with critical points marked by red and blue dots, respectively, 
and pre-images of critical values in the corresponding color. 
Always $m_1 = 3/10$, $m_2=7/10$, $d = 1$. 
Top:     $h = -1/6 > h_\nu$,  types S', L, and P occur. 
Middle: $h = -2/3 > h^*$,  types $\mathrm{S'}$, S, and L occur. 
Bottom: $h = -6/5 < h^*$,  type S and S' only. }
\label{fig:CPasy}
\end{figure}

In this section we recall what we need from the detailed description of the solution to the problem of two fixed centers 
given by  \cite{WDR03}.  Many of the results here are refered to in theorem 2. 

Without loss of generality we set $d = 1$  so that the centres are place at $(1, 0)$ and (-1,0). 
Standard confocal elliptic coordinates $(\lambda , \nu)$ are then defined by the relation $x + i y =  \cosh( \lambda + i \nu)$.
We will slightly alter the standard relation, defining instead
\[ x + i y = \sin( \nu + i \lambda) \,.
\]
In terms of real and imaginary parts we have
$$(x,y) = ( \cosh \lambda  \sin \nu,  \sinh \lambda \cos\nu) $$ 
Combined with  the time reparameterization $\tau = \tau(t; \lambda, \nu)$ given by  
$\dee t = ( \cosh^2 (\lambda) - \cos^2 (\nu)) \dee \tau$, these variables regularize and  separate the  
   Hamiltonian at energy $h$: 
\[
    \tilde H =  \tilde H(\lambda, \nu, p_{\lambda}, p_{\nu} ; h)  := ( H - h) \frac{\dee t}{\dee \tau}  = H_\lambda(\lambda, p_\lambda) + H_\nu( \nu, p_\nu)
\]
into two one-degree of freedom Hamiltonians:  
\[
\begin{aligned}
    H_\lambda( p, q ;h)  & = \frac{ p^2}{2} - ( m_1 + m_2) \cosh q -  h \cosh^2q, \\
    H_\nu( p, q ;h)          & = \frac{ p^2}{2} +  ( m_1 - m_2) \sin q +  h \sin^2 q \,.    
\end{aligned}
\]
We must take the value  of $\tilde H$ to  be  $0$  since we want $H = h$.  Then we have that   $H_\lambda(p,q) = -g$ and $H_\nu(p,q)=g$ 
where $g$ is the value of the separation constant  which is also the second integral $G$, see equation \eqref{G}. (Our $H_{\lambda}, H_{\nu}$ differ from those
in \cite{WDR03} by $-h \cosh^2 (\lambda), h \sin^2 (\nu)$ respectively.)

 The map $(\lambda, \nu) \to (x,y)$ defines a branched  cover of the $(x,y)$ plane, 
 branched over the two centres.    The inverse image  of the $x$ axis ($y = 0$) consists of the 
 lines  $\lambda = 0$ and $\nu = \pm \pi/2 + 2k \pi,  k\in \Z$.   The intersections of these lines,
 $(\lambda, \nu) = (0,  \pm \pi/2 + 2k \pi)$ get mapped in an alternating way to the two centres.
 The inverse image of window  3 (the segment between the two centers) corresponds to   the line $\lambda = 0$. 
The inverse image of window   1 corresponds to the lines $\nu = -\pi/2 + 2 k \pi$.
The inverse image of window    2 corresponds to the lines  $\nu = \pi/2 + 2 k \pi$.
It is central to    our analysis that the symbols 1, 2 and 3 are defined by   coordinate lines in the separating  variables.

The energy has a single critical point lying on window 3 corresponding to an unstable equilibrium lying on the $x$-axis at the 
point where the forces exerted by the two centres balance.  The associated critical value of energy  is
\begin{equation} \label{hstar}
    h_*= - ( \sqrt{ m_1} + \sqrt{m_2})^2/2
 \end{equation}
and serves as a bifurcation value.  
{\it We view $\nu$ as an angular coordinate} so that $(\lambda, \nu) \in \R \times S^1,  S^1 = \R/ 2 \pi \Z$. Then the  Hill region  for energy $h$ is  the domain
\[ 
   \text{ Hill}(h)= \{(\lambda, \nu):  \exists p_{\lambda}, p_{\nu} \text{ such that } \tilde H(\lambda, \nu, p_{\lambda}, p_{\nu} ; h) = 0 \} \subset \R \times S^1 \,.
\]
 For $h < h^*$
the Hill region consists of two disjoint disks, one disc about each centre.
These   discs merge at the equilibrium for  $h = h^*$.  For 
 $0 > h > h_*$   the Hill region is a connected domain, topologically an annulus $S^1 \times I$,   wrapping once around the 
 cylinder.   
 
 See Fig.~\ref{fig:CPsym} and  ~\ref{fig:CPasy} for phase portraits of the two one-degree of freedom Hamiltonians for 
 various values of the masses. In these figures we make a number of contour plots
 in order to indicate   how the 
 one-degree-of-freedom systems  combine to describe
the two-degree-of-freedom system.  
 
 We now collect known facts about the critical values of the one-degree of freedom Hamiltonians
 and the  implications these facts  have for the critical values of the integral map $(G, H)$,
 see \cite{WDR03} for details.

 Throughout we impose the condition  $h < 0$ since if $h  \ge 0$ all motions are unbounded. 
We set 
\begin{equation} \label{hcrit}
h_{\lambda} = -(m_1 + m_2)/2  \quad \text{ and } \quad h_\nu = - |m_1 - m_2|/2 ,
\end{equation} 
and  
\begin{equation}
 \label{kappa}
\begin{aligned}
  \kappa_{\pm\sigma}(h) & =   h + \sigma   | m_1 \pm m_2|, \quad \sigma = \pm \\
  \chi_\pm(h) & = -\frac{ (m_1 \pm m_2) ^2}{ 4h } \,.
\end{aligned}
\end{equation}
Here the first subscript $+$ refers to critical values of $H_\lambda$, and the $-$ to critical values of $H_\nu$.

{\sc Critical values of $H_{\lambda}$}.   For $h < h_{\lambda}$ the function $H_{\lambda} (\cdot, \cdot, h)$
has exactly  one critical value,    $-\kappa_{++}$. 
For $h_{\lambda} < h$  the function has exactly  two critical values
$-\kappa_{++}$ and   $-\chi_+$.  

 {\sc Critical values of $H_{\nu}$}. For $h_{\nu} < h$ the function   $H_\nu(\cdot, \cdot, h)$ has  exactly two critical values $\kappa_{--}$  and $\kappa_{-+}$.
 For $h < h_{\nu}$ the function  has three critical values $\kappa_{--}$, $\kappa_{-+}$,
 and   $ \chi_-$.
 
{\sc Critical and Regular values of the integral map.} We can
deduce the critical and regular values of the integral  map $(G, H) : T^*(\R \times S^1) \to \R^2$ immediately from the above description of the
critical values  of the one-degree of freedom Hamiltionians  and the separation of variables.  
  The critical values of the map
consists of the  union of  either four or five
 analytic curves depending on the values of the masses.  {\bf See figure ~\ref{fig:EM}}. Two or three of these curves are straight lines with slope 1.  The remaining  two
 curves are arcs of hyperbolas.   
  The lines  are  
$g =  \kappa_{++}(h)$ , $g =   \kappa_{-+}(h)$, and $g = \kappa_{--} (h)$. 
The  arcs of hyperbolas are given by 
$g =   \chi_+ (h)$ for $h_\lambda \le h $
and 
$g =  \chi_- (h)$ for $h^* \le h < h_\nu$.
The image of the  integral map is the region bounded between the leftmost and rightmost of these curves  (remember: $h  < 0$  always!).  This image is divided by the curves
into three or four   simply connected curvilinear domains 
whose interiors consist of the map's  regular values. These regions, {\it henceforth called the ``regular regions''} are 
denoted  by the  symbols 
 S', S, L, P as indicated in the figures. 
The preimage of a  point in   region  S or P is  two disjoint 2-dimensional tori,
while if the point is in  S' or L this  preimage is   a single   torus.
The terminology goes back to Charlier \cite{Charlier02} and Pauli \cite{Pauli22}. S and S' stands for satellite, L for Lemniscate, and P for planetary.

\begin{figure}
\centering{ 
\includegraphics[width=7cm]{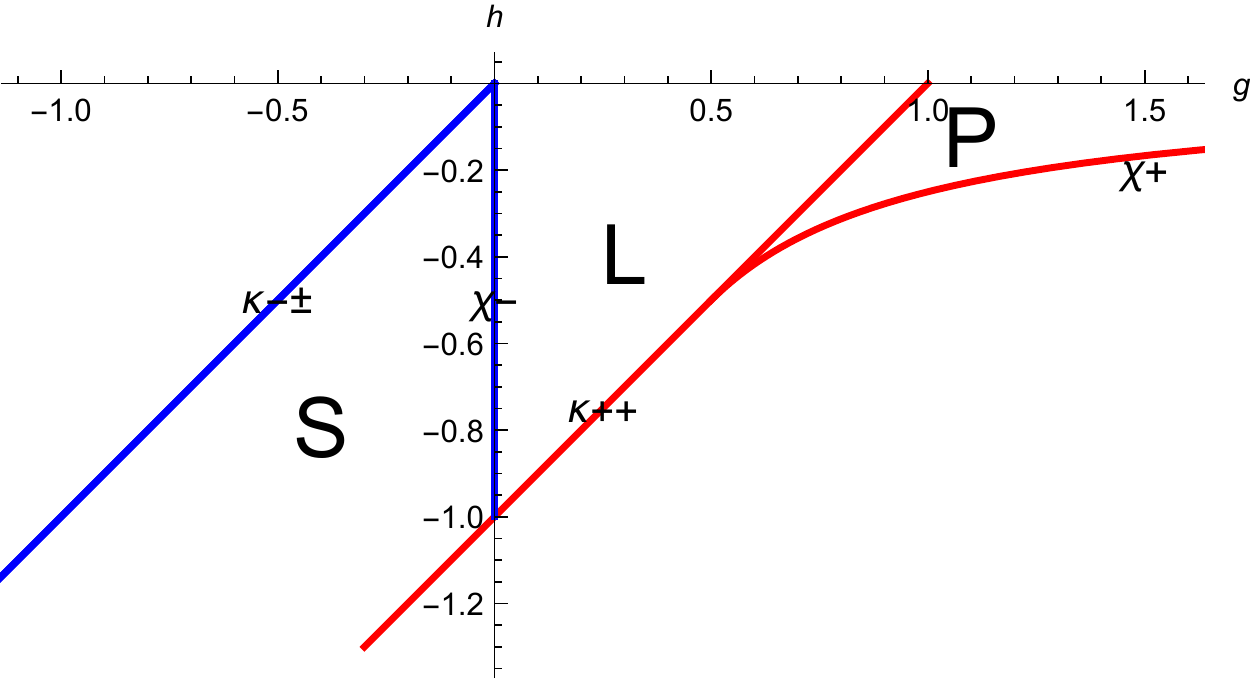}
\includegraphics[width=7cm]{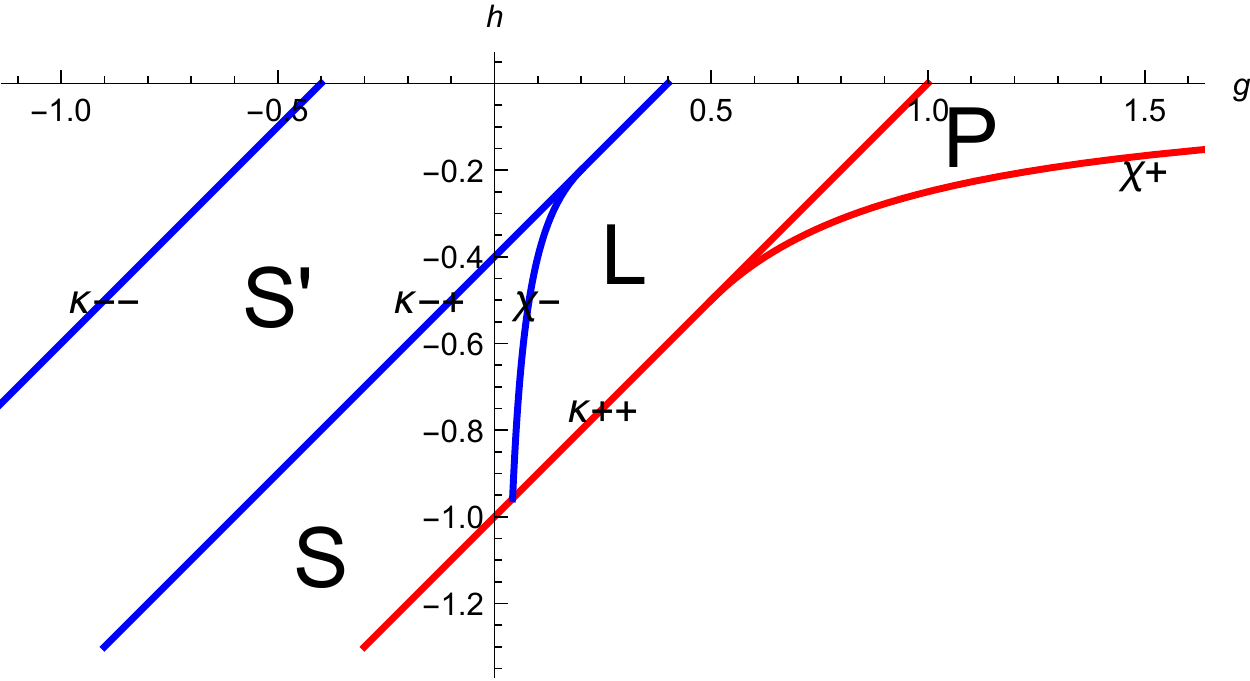}
}
\caption{Critical values of the energy momentum map $(G, H)$ 
in the symmetric case $m_1 = m_2 = 1/2$ (left)
and in the asymmetric case $m_1 = 3/10$, $m_2 = 7/10$ (right). 
The regions of regular values are marked by S, L, P, 
for satellite, lemniscate, and planetary type, respectively.
Blue lines are critical values of $H_\nu$;  red lines are critical values of $H_\lambda$.
}
\label{fig:EM}
\end{figure}

Our  one-degree of freedom systems  
have  period functions $T_\lambda(g,h)$ and $T_\nu(g,h)$ defined by an integral over the curve
$H_\lambda (\cdot, \cdot , h) = -g$ and  $H_\nu(\cdot, \cdot , h)= +g$.
These period functions  are analytic
within each regular region of the $(g,h)$ plane.  
Our tori $\T(g,h)$ admit a natural homology basis, or coordinate system, 
corresponding to our separation of variables.  We always define the rotation
number $W$ in terms of this basis.
Then our rotation number is
\begin{equation}
\label{RotnNumber}
W (g,h) = T_{\nu} (g,h)/ T_{\lambda} (g,h)
\end{equation}
which is a piecewise analytic function.   
We give explict formulae for $W$ further on.  
One surprise is that $W$ only depends on the value $(g,h)$ of the integral map in the case S where there are two tori for a given $g,h$:
the value of the rotation number on these two tori  is the same. 
Another surprise is that $W$ is analytic across the curve separating the region S from S'. 

See Fig.~\ref{fig:TW} (bottom) for a plot of $W$ for  energy $h = -1/4$ in the symmetric 
case. For this energy all three types S, L, and P occur.

The function $T_{\lambda}$ has   discontinuity only along the red singular curves of figure~\ref{fig:EM} 
so  has two `branches'  denoted 
$T_{\lambda 3}$ and $T_{\lambda 0}$ the first having   domain   formed by the union of S,  S'  and L, while the  second branch has   domain   P.  
(In case S'  is empty the first domain is just S union L.)
The function $T_{\nu}$ has its only discontinuity along the blue  singular curves of figure~\ref{fig:EM}  
so   also has two branches,   denoted $T_{\nu o }$ and $T_{\nu r }$, the first branch having  
domain the union of S  and  S'  and  the second branch $T_{\nu r }$ having domain  the union of P  and   L. 
See figure~\ref{fig:Tlevels}.

Consequently, the three branches of $W_{S, L, P}$ of the   rotation number are given by:  
\[
    W_S(g,h) =  \frac{T_{\nu o}(g,h)}{T_{\lambda3}(g,h)}, \quad
    W_L(g,h) =  \frac{T_{\nu r}(g,h)}{T_{\lambda3}(g,h)}, \quad
    W_P(g,h) =  \frac{T_{\nu r}(g,h)}{T_{\lambda0}(g,h)}\,,
\]

For us, the crucial property of the rotation number $W$ is its range. If we  fix  $h$,
and let $g$ vary, the rotation numbers   sweep out  an interval  
  $[W_{min} (h),  W_{max}(h)]$  which have the following dependence on   energy  $h$. \\
 For type L:
 \begin{equation}
   \label{WtypeL}
\begin{array}{c|cc}
h                                        & W_{min}              & W_{max}   \\ \hline
0 > h > h_\lambda            & 0                                &   \infty    \\
h_\lambda > h > h^*        &  W_L(\kappa_{++},h)  &  \infty   \\
\end{array}
\end{equation}
For type S:  
 \begin{equation}
   \label{WtypeS}
\begin{array}{c|cc}
h                                   & W_{ min} &  W_{ max}   \\ \hline
0 > h > h^*             &  W_S(\kappa_{--}(h),h) &\infty   \\
h^* > h                       &  W_S(\kappa_{--}(h),h) & W_S(\kappa_{++}(h),h) 
 \end{array}
\end{equation}
where the critical values $\kappa_{++}(h), \kappa_{- +}(h), \kappa_{--}(h)$ and $h_\nu$, $h_\lambda$ were defined 
in equations \eqref{kappa} and \eqref{hcrit} above.

\begin{figure}
\includegraphics[width=8cm]{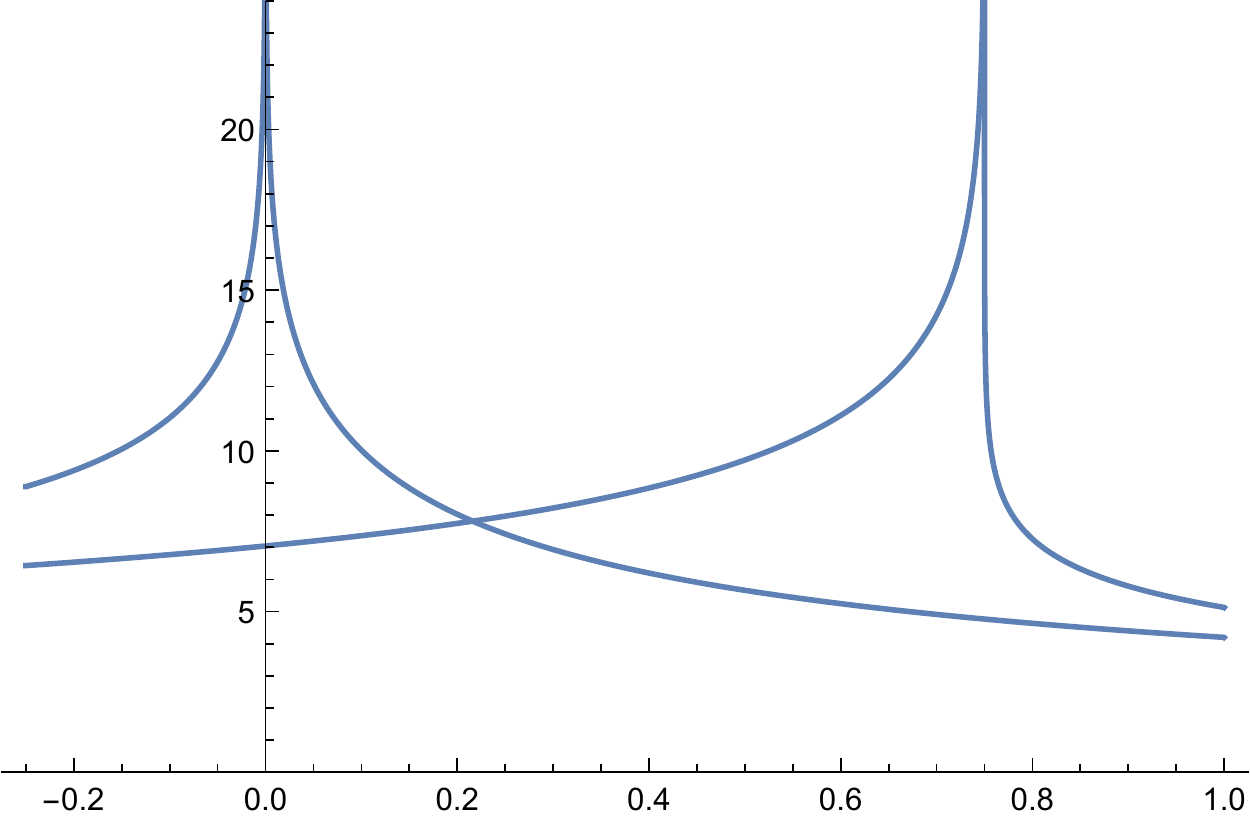} \\
\includegraphics[width=8cm]{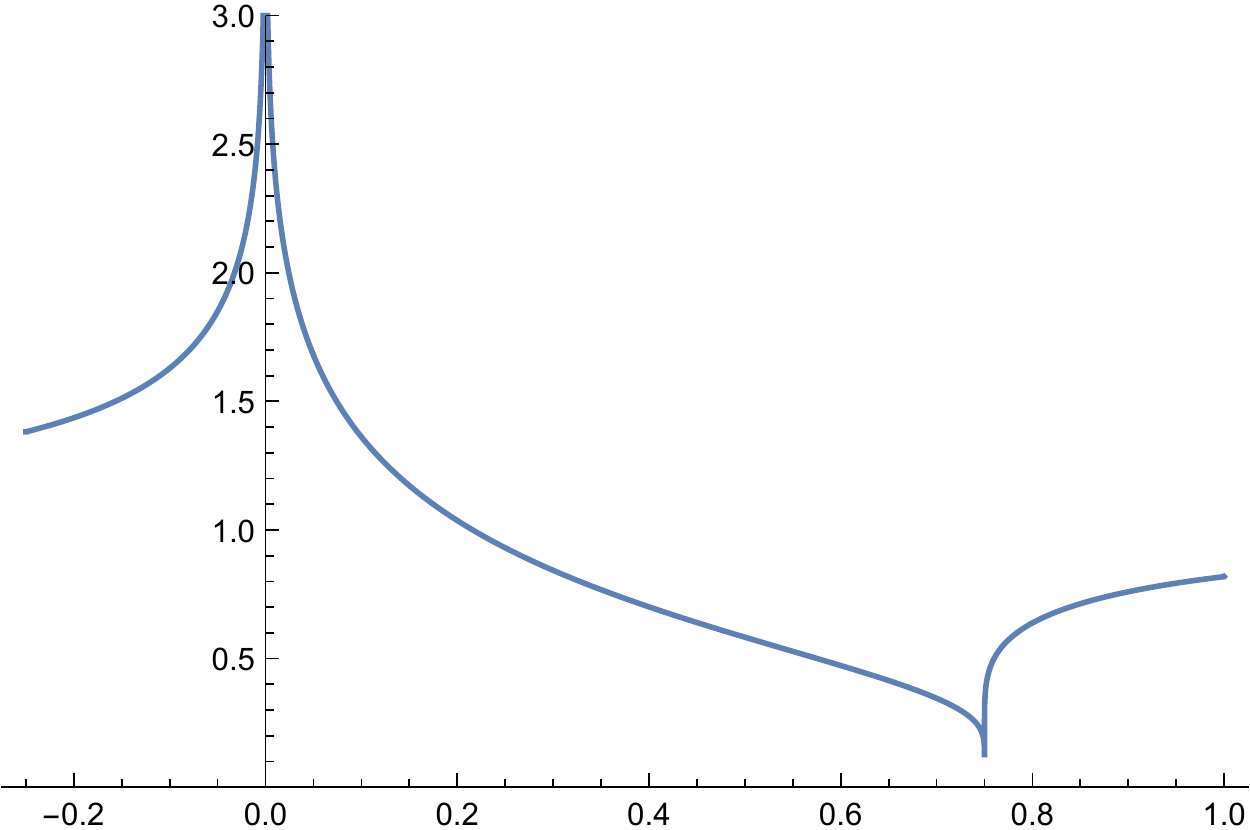}
\caption{The period $T_\nu$ (diverging at 0) and $T_\lambda$ (diverging for positive $g$) (both top) and 
the rotation number $W = T_\nu / T_\lambda $ (bottom)
all for $h = -1/4$, $m_1 = m_2 = 1/2$, $d=1$. 
The type of corresponding tori is S, L, P from left to right, separated by singularities.}\label{fig:TW}
\end{figure}

\begin{figure}
\includegraphics[width=7cm]{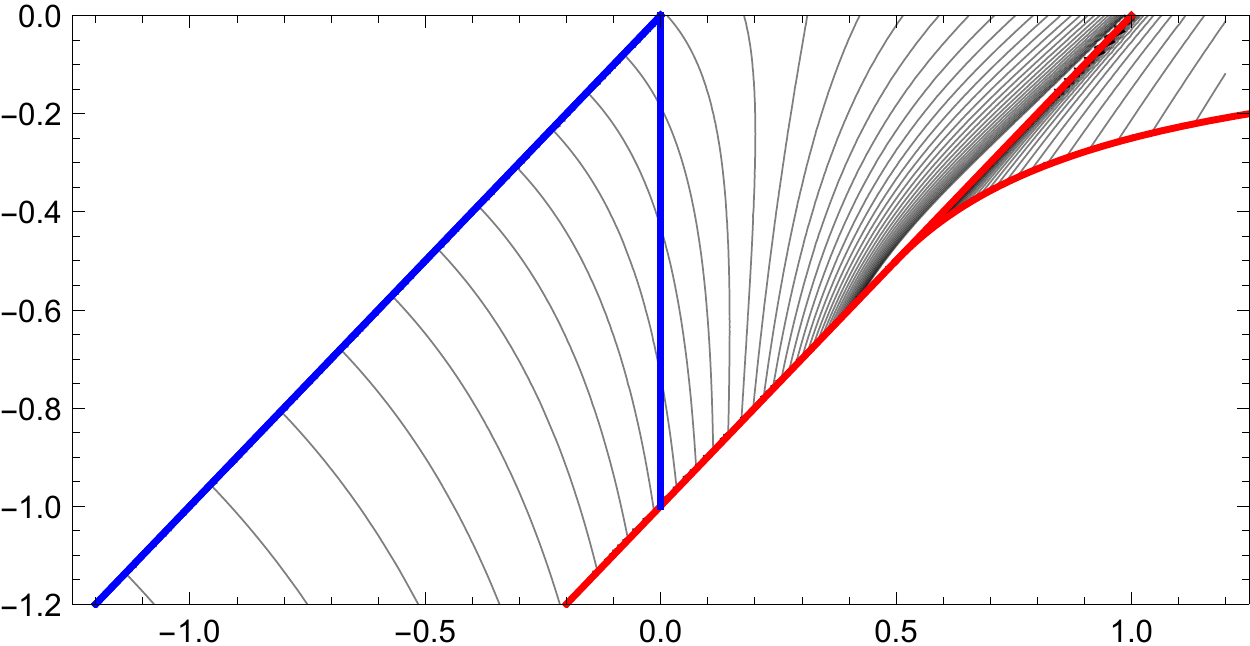}
\includegraphics[width=7cm]{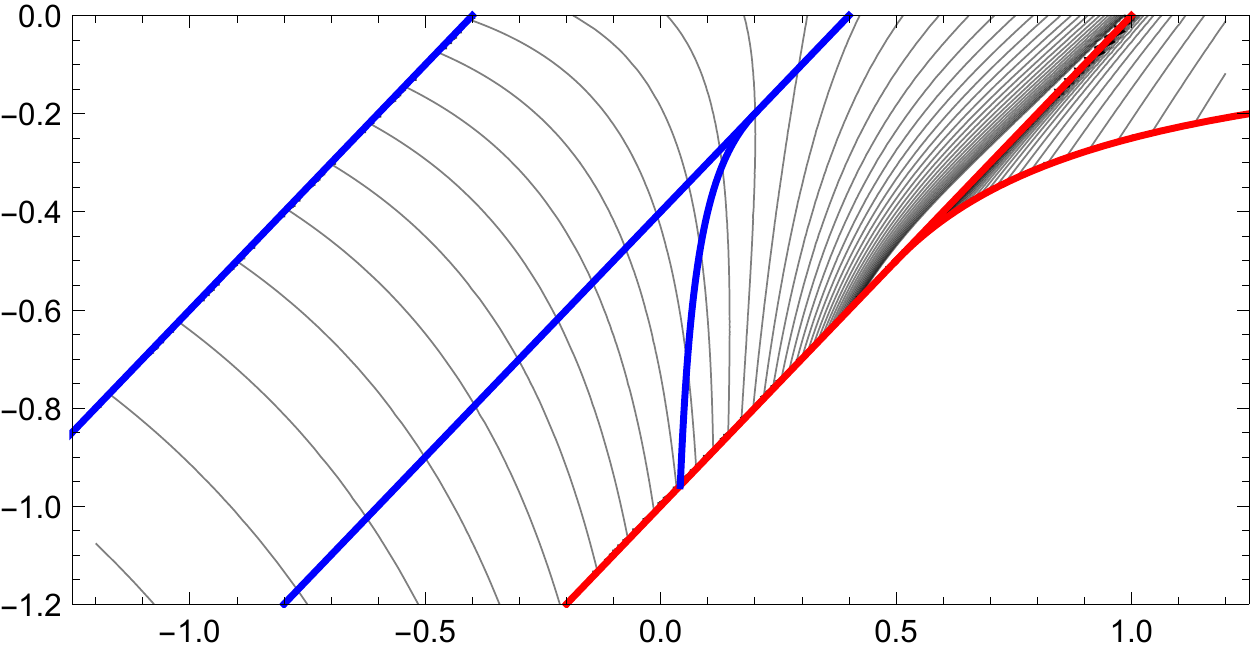}
\includegraphics[width=7cm]{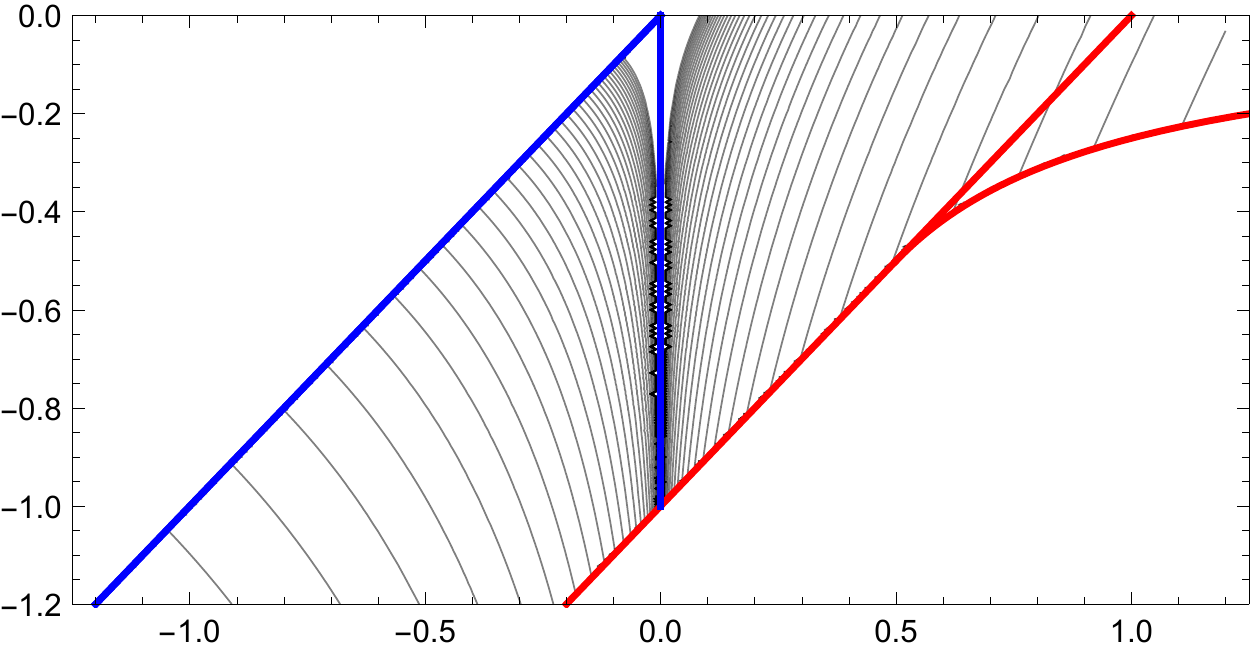}
\includegraphics[width=7cm]{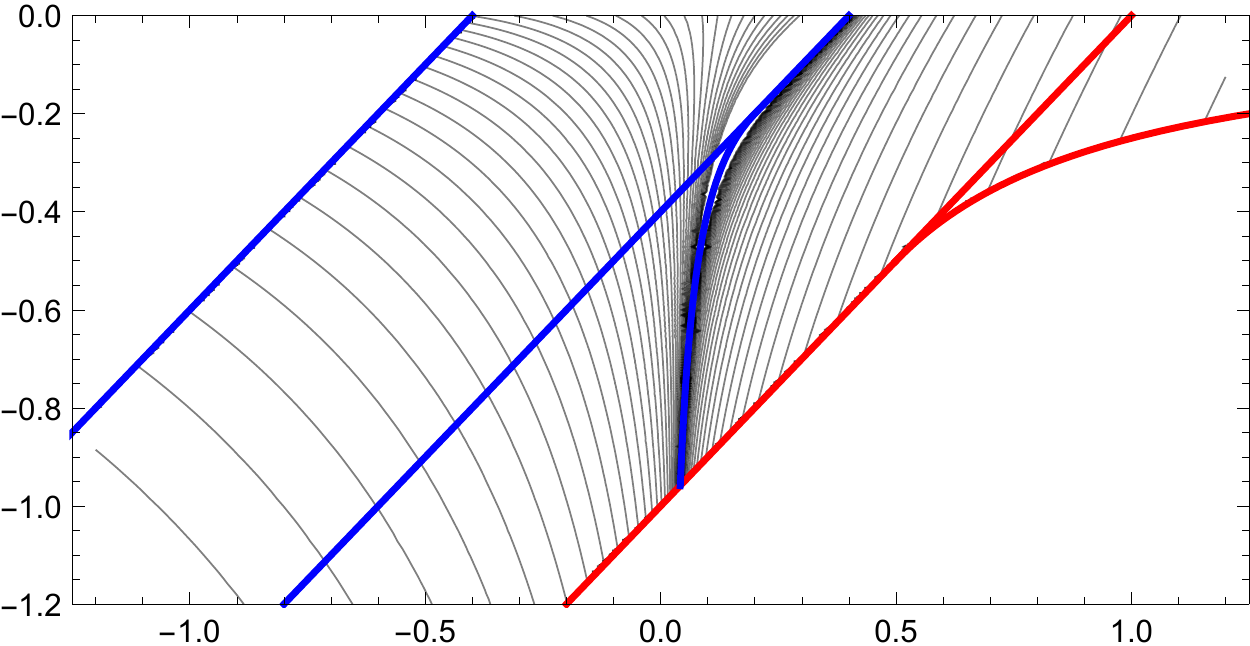}
\caption{Contours of constant period $T_\lambda$ (top) and $T_\nu$ (bottom) 
for the symmetric case $m_1 = m_2 = 1/2$ (left) and an asymmetric case $m_1 = 3/10$, $m_2 = 7/10$ (right).}
\label{fig:Tlevels}
\end{figure}

We now proceed to determine the  functional forms of the    period functions, and hence of the rotation number.
We have  summarized their forms    in Fig.~\ref{fig:Tlevels}.  We use  the standard method of period computation from the classical mechanics
of one-degree of freedom   systems (e.g.~\cite{LandauLifshitz}). 
Let $\Psi(q,p)$ denote either  $H_\nu (q,p; h)$ or $H_\lambda (q,p;  h)$.   $\Psi$ has  the form
$\frac{1}{2} p^2 + V(q, h)$. The first of Hamilton's two equations reads $dq/d \tau = p$ so that $d \tau = \frac{dq}{p}$.
Solutions must lie on  energy curves $\{ \Psi = c\}$, $c = \pm g$,  which  consist for us  of   one  or two   topological circles.  
Choose a component  $C$ for a given constant  $c$. The  time $T$  to traverse $C$ is the corresponding period $T$ which  we want to compute.
This time  is  
given by the integral 
$$T = \int_C d \tau = \int_C \frac{dq}{p}.$$
(One can solve for $p$ in terms of $q$ : $p = \pm \sqrt{ 2(c - V(q;h)}$
and use the time-reversal symmetry to rewrite this as
$T = 2 \int_{q_{min}} ^{q_{max}} \frac{dq}{\sqrt{ 2(c - V(q;h)}}$.)
In our   cases the variable $q$ is either $\nu$ or $\lambda$ and the variable  $p$ is either  $p_{\nu}$ or $p_{\lambda}$
The substitution  $z = \sin \nu$ or $z = \cosh \lambda$  converts the integrand 
  $\frac{dq}{p}$  to the integrand 
 $\frac{dz}{P}$
where  the integral in the new variables is   around the loop    corresponding to our choice $C$ which
lies on  the Riemann surface: 
 $$
      P^2 =  \begin{cases}
      2(1-z^2)(+g + (m_1 - m_2)  z + h z^2 ), \qquad \nu \text{ case}   \\
      2(1-z^2)( -g + (m_1 + m_2)  z + h  z^2), \qquad \lambda \text{ case } \\
\end{cases}
$$
The integral $\int \frac{dz}{P}$ over  the  closed loop is a complete elliptic integral which  can be expressed in terms of 
Legendre's complete elliptic integral $\K(k^2)$ with modulus $k$.

For the symmetric case $m_1 = m_2$ we find
\[
   T_{\nu o}(g) =  \frac{4}{ \sqrt{ - 2 h} } \K( k^2),  \text{ for } g < 0 , \quad  
\]
\[   
   T_{\nu r}(g) =   \frac{4}{ k \sqrt{ - 2 h} } \K(1/k^2),  \text{ for} g > 0, 
\]
\[
\text{ where } 
   \quad k^2 = 1 - \frac{g}{ h}, 
\]
Expressions of elliptic integrals in terms of Legendre's standard integrals can, e.g., be 
found in \cite{BF71}.
Legendre's $\K$ is a smooth monotonically increasing function that maps $k^2 \in ( -\infty, 1)$
to $(0, \infty)$. 

To write down the period $T_\lambda$ and to treat the asymmetric case $m_1 \ne m_2$ introduce 
\[
     k_\pm^2 =  \frac12 + \frac{g +  h }{ 2 \sqrt{ 4 g h + (m_1 \pm m_2)^2}}, \quad
     f_{\pm\sigma} = \frac{\sqrt{2}}{\sqrt{ \sigma(-1 - \frac{g}{h}) - \frac{\sqrt{4 g h  + (m_1 \pm m_2)^2}}{ h} }}
\]
where $\sigma = 0$ or $\sigma = 1$.
Then the $\lambda$-periods are
\[
   T_{\lambda3}(g,h) =  \frac{4}{ \sqrt{ 2  |h|} }f_{+0} \K( k_+^2 ), \text{ for } g < \kappa_{++}(h) 
\]
\[
   T_{\lambda0}(g,h) =  \frac{2}{ \sqrt{ |h| }}f_{+1} \K(1/k_+^2),  \text{ for }  \kappa_{++}(h)  < g <  \chi_+ (h)
\]
When $h > h_\lambda$ the modulus $k_+^2$ is monotone in both cases,
increasing and decreasing, respectively.
Now $f_{+0}(g)$ is always monotonically increasing, so that $T_{\lambda3}$ is monotonically 
increasing in this energy range. 
%

Note that since without loss of generality we can set $m_1 + m_2  = 1$,
the  $T_\lambda$ periods do not change with parameters, only their domain of 
definition in the integral image changes. In particular the periods $T_\lambda$ 
in the symmetric case $m_1 = m_2$ and the asymmetric case $m_1 \not = m_2$ are 
given by the same functions.

For oscillations of $\nu$ (when $m_1 \not = m_2$) the period is
\[
   T_{\nu o}(g,h) =  \frac{4}{ \sqrt{  2  |h|} } f_{-0}  \K( k_-^2 )
\]
where for $h > h_\nu$ the domain is $g < \kappa_{-+}(h)$ and the modulus is monotone,
while for $ h < h_\nu$ the domain is $g < \chi_-$ and $k_-^2$ passes through 
zero at $ h \pm (m_1 - m_2)$.
The formula is valid in region S and in region S', in region S it gives the 
period for both tori of type S, around either center.

For rotations of $\nu$ (when $m_1 \not = m_2$) there are two cases depending on whether 
the elliptic curve has 2 real or 4 real branch points.
A formula that is valid in both cases is defined via
\[
  k_c^2 = \frac12 - \frac{ g +  h}{2 \sqrt{ (g -  h)^2 -  (m_1 - m_2)^2 } }
\]
and
\[
  T_{\nu r}(g,h) = \frac{4}{\sqrt{2} ( (g- h)^2 -  (m_1- m_2)^2 )^{1/4} }  K( k_c^2 ) \,.
\]
When there are four real branch points the sign of $k_c^2$ is negative, 
otherwise $k_c^2$ is positive.
Moreover, when $h < h_\nu$ the modulus $k_c^2$ is monotonically decreasing. 
When $h > h_\nu$ the modulus is monotonically increasing from $-\infty$ to 0,
which occurs at the boundary of the Lemniscate region
for $g = \kappa_{++}$.

The functions $T_{\nu o}(g,h)$ and $T_{\nu r}(g,h)$ are monotone functions of $g$, for $h$ fixed,  because $K(k^2)$ is monotonically increasing, 
and $k^2(g)$ is linear in $g$ and increasing for $h < 0$. So $T_{\nu o}(g)$ is monotonically increasing and diverges for $g \to 0_-$.
Now $1/k^2(g)$ is monotonically decreasing for $g > 0$ and so as a function of $g$,  $T_{\nu r}$ is the product of 
two monotonically decreasing functions, and hence monotonically decreasing. It diverges for $g \to 0_+$.

\noindent {\bf Conjecture.} The   period function $T_{**} (g,h)$ and rotation function $W_{*}(g,h)$ are  monotone functions of $g$ within  
 their domains of analyticity. 

Numerical experiments support  this conjecture.

Note that in the S-region the modulus $k_-^2$ is negative, and approaches $-\infty$ at 
the boundary to L.
Similarly, in the upper parts of the L- and P-region the modulus 
$k_c^2$ becomes negative, and on the boundary between L and S' 
it approaches $-\infty$.

\begin{figure}
\includegraphics[width=12cm]{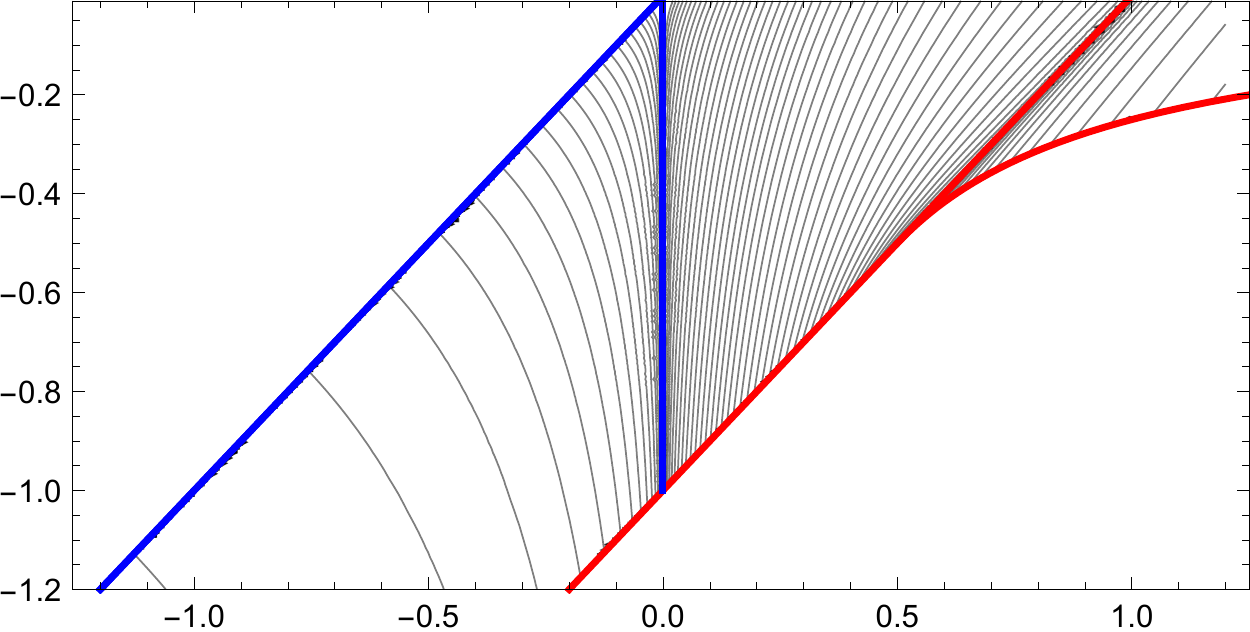}
\includegraphics[width=12cm]{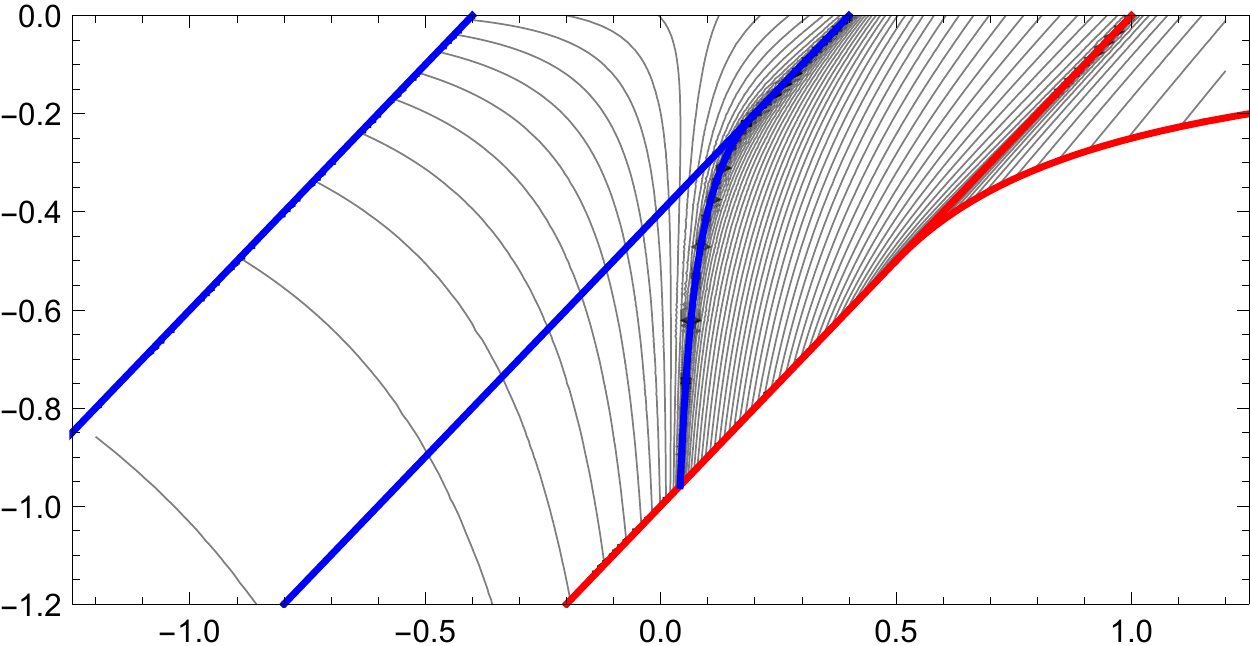}
\caption{Contours of constant rotation number $W$
for the symmetric case $m_1 = m_2 = 1/2$ (top) and an asymmetric case $m_1 = 3/10$, $m_2 = 7/10$ (bottom).
The spacing of contours is even in $W$ or $1/W$, whichever is smaller.}
\label{fig:Wlevels}
\end{figure}

\begin{figure}
\includegraphics[width=10cm]{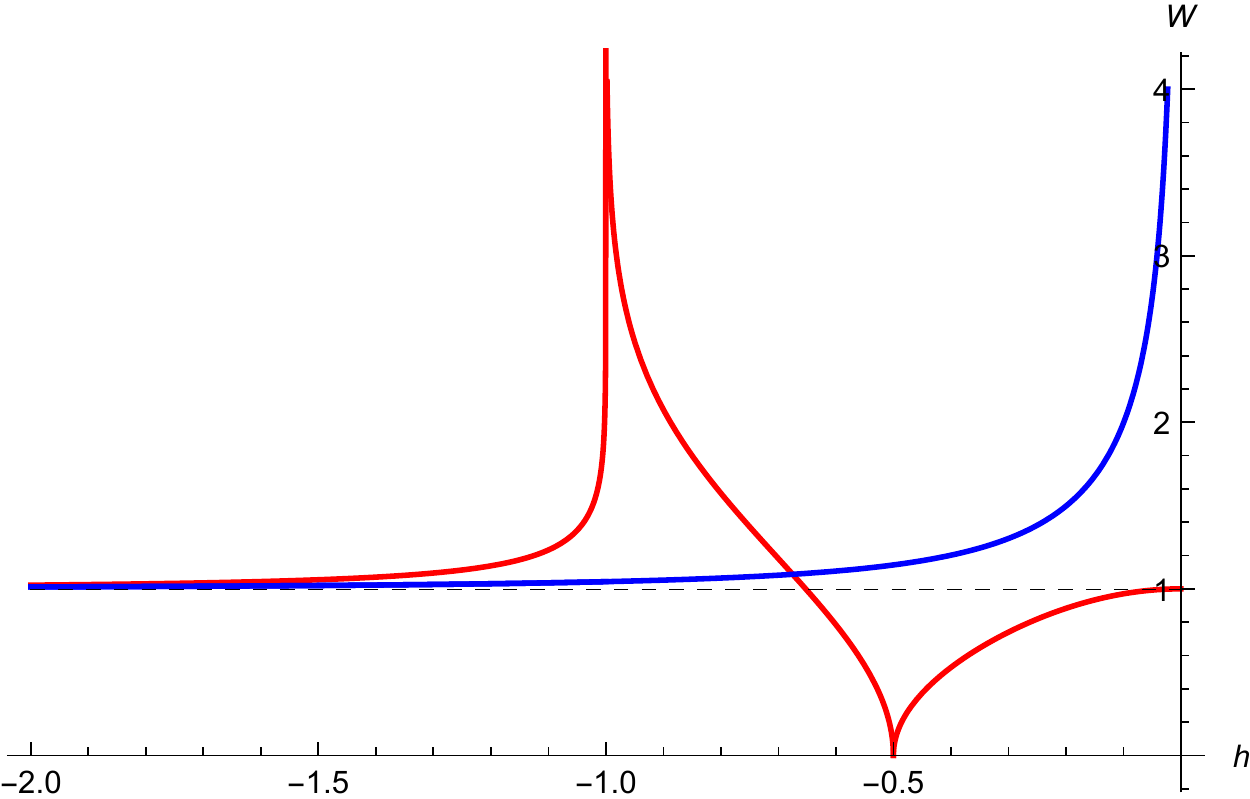}
\caption{The boundary values of the rotation number $W$ along 
the leftmost curve of critical values (blue, see Fig.~\ref{fig:EM}) which gives the minimum in S, 
and along the rightmost curve of critical values (red, see Fig.~\ref{fig:EM}). 
The three smooth parts of the red curve
correspond to the rightmost boundary values of tori of type S, L, and P, 
in that order. The red curves give the maximum in S, minimum in L, and maximum in P, respectively.
Parameters are $m_1 = m_2 = 1/2$, $d=1$.}\label{fig:WB}
\end{figure}


The form of the period functions yield the following facts about
the rotation number $W$.  
For  $h_\lambda < h < h_\nu$ the rotation number $W_L(g)$ is a monotonically decreasing function.
At the left boundary of the L region $W_L$ diverges to $+\infty$, while
at the right boundary of the L region for $h > h_\lambda$ the rotation number $W_L$ approaches 0.
For $m_1 = m_2$ the limiting values of $W$ at the other families of critical values are as follows, 
see figure~\ref{fig:WB} for graphs of these functions.

To obtain   tables \ref{WtypeL}    and \ref{WtypeS}  we need the minimum and maximum of $W(\cdot, h)$
within its various  intervals of   analyticity. 
For simplicity,  in the following formulas
we have set $m_1 = m_2 = 1/2$).  We compute  that 
\begin{itemize}
\item the minimum of $W_S( \cdot , h)$ is : 
\[
  W_S(\kappa_{--}(h),h) = \frac{\pi  \sqrt[4]{4  h^2+1}}{2 \sqrt{-2 h} \, \K\left(\frac12+\frac{ h}{\sqrt{4  h^2+1}}\right)}
\]
\item the maximum of $W_S( \cdot , h)$ is,  if finite, for $h < h^* $: 
\[
  W_S(\kappa_{++}(h),h) = \frac1\pi \sqrt{\frac{4  h + 2}{ h}} \, \K\left(-\frac{1}{ h}\right)
\]
\item the minimum of $W_L (\cdot, h)$ for   $h^* < h < h_\lambda $ occurs at 
\[
  W_L(\kappa_{++}(h),h) = \frac1\pi   \sqrt{-4  h-2} \, \K(- h)
\]
\item the  maximum of $W_P( \cdot, h)$ for   $h_\lambda < h < 0$ is: 
\[
W_P( \chi_+ (h),h) = \frac2\pi  \sqrt{\frac{1 - 4  h^2}{1 + 4  h^2}} \,  \K\left(\frac{1}{1+\frac{1}{4  h^2}}\right)
\]
\end{itemize}


In arriving at these facts we used that for the Lemniscate case 
 $T_{\nu r}(g,h)$ is increasing with $g$ and that $T_{\lambda3}(g,h)$ is decreasing with $g$,  
so that  $W_L(g,h )$ is decreasing with $g$.  
We also used that 
the period $T_\nu$ diverges at the left boundary of the L-region, 
while the period $T_\lambda$ diverges at the right boundary of the L-region when $h > h_\lambda$, 
hence we arrive at the  limiting behaviour of $W_L$. 
For a periodic orbit of type L or S with rational $W = p/q$ with co-prime $p$ and $q$
the period of the periodic orbit in original coordinates is $(p T_\lambda + q T_\nu ) /2$ where the factor $1/2$ 
accounts for the double cover.





\section{Finishing up the proof of theorem 2.}
\label{halfaunit}

We have established the range of the   rotation numbers in the previous section.
We established almost all of the proof in section \ref{guts}.  
All that  remains  to prove is that  the windows are situated as  in figure~\ref{fig:windows}.  
In particular, we have  not yet shown that the  line pairs are spaced   $1/2$ a unit apart. 

We saw that   the two-center  system separates into two  one degree of freedom systems,
that the separate torus coordinates correspond to these separate systems, with $\lambda, p_{\lambda}$ corresponding to the the $\theta_2$ motion
and $\nu, p_{\nu}$ to $\theta_1$-motion. 
The windows  on a given torus are vertical or horizonatl lines in the angle coordinaes
since   the collinear line  $x =0$ corresponds to $\lambda = 0$ or $\nu = const$ and
since the axes of the torus correspond to one or the other of the
separating coordinates:  $\theta_1 = \theta_1(\nu, p_{\nu}; g,h)$, 
$\theta_2 = \theta_2(\lambda, p_{\lambda}; g,h)$.  It follows that windows
have the form 
  $\theta_1 =const$ or $\theta_2 = const$. What constants?
  In other words, 
upon being projected onto either  one-degree of freedom system, the   windows correspond to  points (or the empty set).
   But which points and how are they separated? 

The explicit maps $\theta_1, \theta_2$ as functions of  mapping from the regularised coordinates $(\nu, \lambda, p_\nu, p_{\lambda})$ 
are   complicated expressions  in terms of    elliptic functions. 
Fortunately, this detailed information is not needed to answer our question  due to the presence   discrete symmetries  for the separated Hamiltonians.

The symplectic involution
$R_\lambda(\lambda, p_\lambda) = (-\lambda, -p_\lambda)$
leaves $H_\lambda$ invariant and hence commutes with the time evolution for that flow. 
The symplectic involution  
$R_\nu(\nu, p_\nu) = (\pi - \nu, -p_\nu)$
leaves $H_\nu$ invariant and so commutes with its time evolution 
(If $\tilde R_\lambda, \tilde R_\nu$ denotes the extension of these involutions
to the full regularized plane: for example, $\tilde R_\nu(\lambda, \nu, p_{\lambda}, p_\nu)  =  (\lambda, \pi - \nu,p_{\lambda}, -p_\nu)$
then the composition of the two involutions $ \tilde R = \tilde R_\nu \circ \tilde R_\lambda$ 
 is an involution on the double cover that leaves
the original point $(x,y,p_x, p_y)$ fixed.)
Rewritten in terms of angle coordinate, $\theta = \theta_1$ or $\theta_2$,
either $R$ must commute with translation (since it commutes with its Hamiltonian flow) and thus is of the form $R(\theta) = \theta + p$
for some constant $p$. But since each $R$ is a nontrivial involution we must have
that $p = 1/2$.

For the $\lambda$ motion these window points are
\begin{itemize}
\item[S',S,L:] the  two points   $(0, p_\lambda)$, $(0, -p_{\lambda})$ for the  symbol 3.
exchanged under $R_\lambda$. (Note that $p_{\lambda} =0$  above corresponds to the purely linear motion, and a degenerate torus,
which we have exculded from considerations.) 
\item[$P:$]  No symbol occurs.
\end{itemize}
As we just saw, $R_{\lambda}$ has the form
$R_{\lambda} (\theta_2) = \theta_2 + 1/2$, proving that the corresponding window points a half a unit apart,
and so the corresponding lines are half a unit apart. 
 In other words, the window for the symbol `3' consists of   two horizontal lines   
separated by half a lattice unit.  

For the $\nu$ motion the window points   are
\begin{itemize}
\item[S1:]  two points of the form $(\pi/2, \pm p_\nu)$ for the  symbol 1,
exchanged  by $R_\nu$.
\item[S2:]   two points of the form $(-\pi/2, \pm p_\nu)$ for the  symbol 2,
exchanged  by  $R_\nu$.
\item[P:] a point $(\pi/2, p_\nu)$ for symbol 1   and a  point $(-\pi/2, \tilde p_\nu)$ for  symbol 2,
where $p_\nu > 0$ and $\tilde p_\nu > 0$.
There are two more such points with negative momenta which belong to a disjoint orbit of $H_\nu$
(``prograde'' or ``retrograde'' motion)
The two groups of points are mapped into each other by $R_\nu$.
\end{itemize}
By an argument identical to the $\lambda$-case, the two window points in each
case are separated by a half, and so the corresponding vertical lines on the torus
are separated by half a unit.  

The  resulting windows on the flattened torus for the three case 
L, S, P are thus as shown in figure~\ref{fig:windows}.

The location of the windows relative to the origin of the flattened torus depends on the choice of origin of the two angles $\theta_1$ and $\theta_2$ and is irrelevant.
The intersection of   windows correspond to   collisions, 
each of which appears twice since the regularisation 
led to a double cover of configuration space.

\begin{figure}
 \includegraphics[width=10cm]{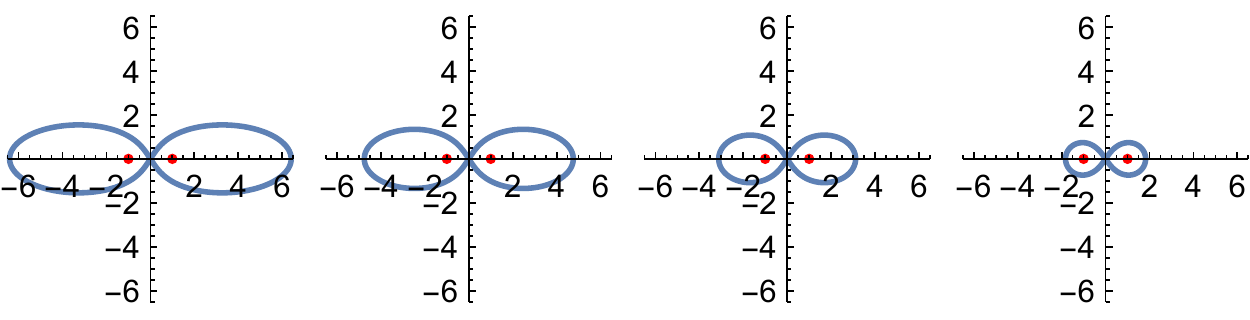}
\includegraphics[width=10cm]{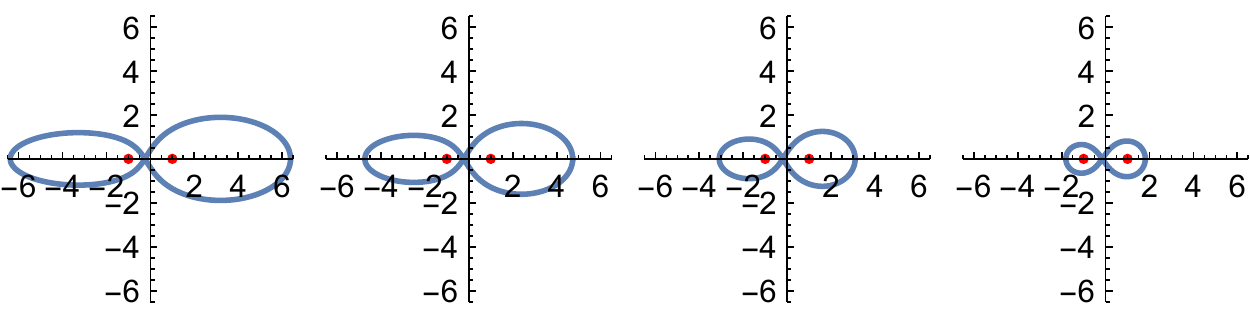}
 \includegraphics[width=10cm]{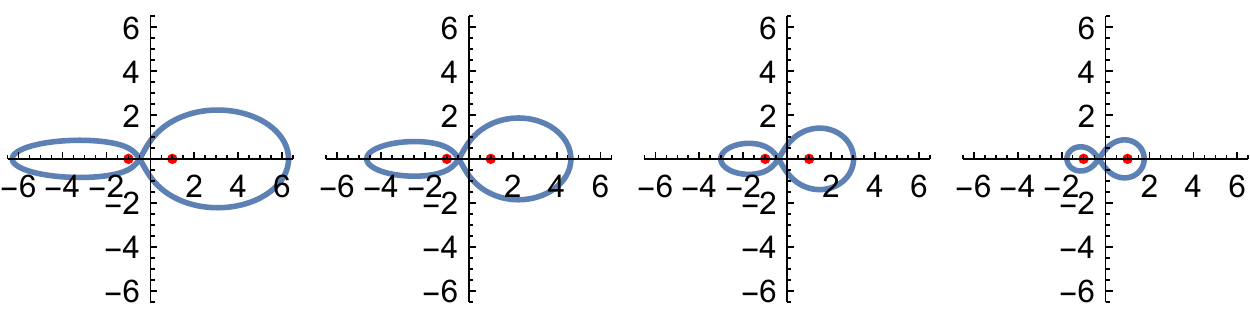}
\includegraphics[width=10cm]{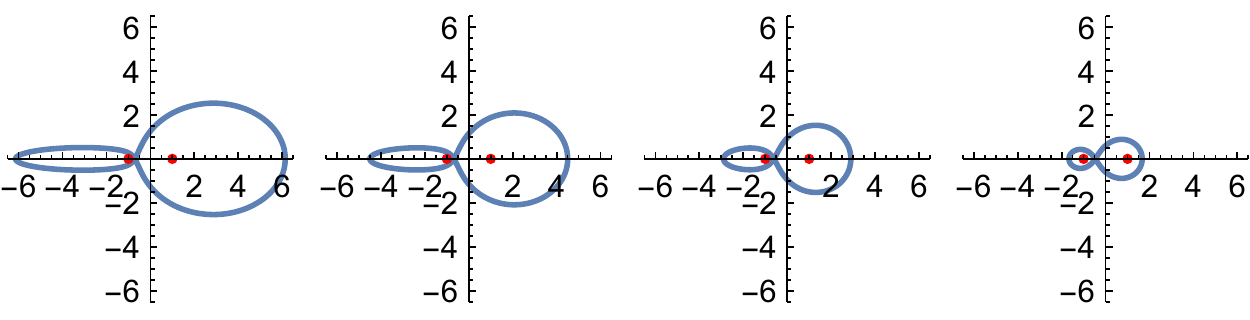}
 \includegraphics[width=10cm]{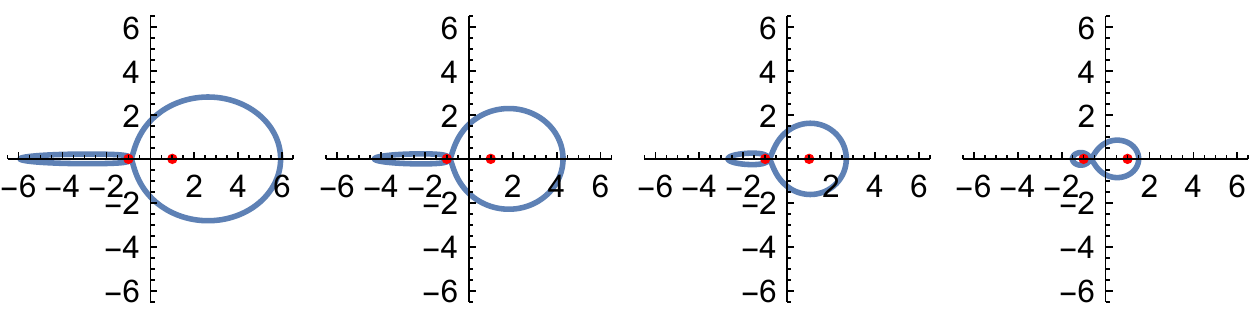}
 \caption{Orbit of type L with $W = 1$ and syzygy sequence $1323$ for 
 various values of $h = -0.15, -0.2, -0.3, -0.5$ from left to right 
 and values of $m_1 / m_2 = 0.5, 0.4, 0.3, 0.2, 0.1$ from top to bottom
 all with $m_1 + m_2 = d = 1$. For type $L$ and $h > h_\lambda = -1/2$ 
 all orbits exist in such 2-parameter families.
 } \label{fig:resilience}
 \end{figure}

\section{Additional Remarks}

%

{\bf Remark}: When $W = p/q$ with $q$ even the symbol sequence of type L
repeats after $p+q$ symbols. When $q$ is odd the 2nd half of the symbol 
sequence is like the first with symbols 1 and 2 exchanged. 
For periodic orbits of type S the two halves of the symbol sequence are always identical.

{\bf Remark}: There are no break-orbits of type L, since for these 
$p_\nu$ is always non-zero. For type S there are  
break-break orbits for $W = p/q$ with $p$ or $q$ even.
When both $p$ and $q$ are odd the corresponding symmetric periodic orbit 
has reflection symmetry about the $x$-axis (but no break points), 
and the corresponding collision orbit is a break-collision orbit.

{\bf Remark}: Discrete symmetry of periodic orbits. Let $W = p/q$. \\
Type L, arbitrary masses: 
If $q$ is odd, there is an orbit that has a crossing of the $x$-axis with a right angle.
If $p$ is odd, there is an orbit that crosses the origin.
In addition, if the masses are equal and $p$ or $q$ are even, 
there is an orbit that has a crossing of the $y$-axis with a right angle. 

Type S, arbitrary masses: 
If $p$ or $q$ are even, there is a break-orbit.
If $q$ is odd there is an orbit that has a crossing of the $x$-axis with a right angle.


{\bf Remark}: A non-obvious consequence of the proof and the description from section~\ref{sec:Euler}   is that 
there are no bifurcations in the orbit structure and the symbol sequences of the L-region when 
changing the mass ratios or the  energy within   the range $(h_\lambda, 0)$.
In particular all periodic orbits persists without bifurcations in this 2-parameter family.
This strong resilience of orbits for sufficiently large energy is quite astonishing. 
It is illustrated for the $W = 1/1$ orbit shown for various energies and various 
masses in figure~\ref{fig:resilience}.
For smaller energies this is not true and orbits bifurcate away on the right boundary 
of region L.

{\bf Remark}: Decreasing $h$  from $h_\lambda$ to $h^*$ (of the saddle equilibrium point) removes more and more 
rotation numbers by increasing the minimum value of the range $W \in ( W_L^{min}, \infty)$, 
thus only allowing orbits with more and more consecutive symbols 3, which describe close approaches to the 
isolated hyperbolic periodic orbit above the saddle.
It is interesting to note that for any $h > h^*$ (even arbitrarily close), the whole 
complexity of orbits generated by $\{ W \} $ passing through $[0, 1)$ is still present, 
in fact there are always infinitely many such intervals. 

{\bf Remark}: For tori of type P the standard syzygy sequences do not separate tori,
because symbol 3 does not occur and symbols 1 and 2 simply alternate for all allowed 
rotation numbers $W$. Instead of insisting on the usual symbols of the 3-body problem 
we can turn the construction around and choose a window on the torus so that we get 
similar symbol sequences with a different physical interpretation. 
As we already pointed out for separable integrable system the windows are best chosen 
as coordinate lines of the separating coordinate system. We choose an ellipsoidal coordinate
line $\lambda = \lambda^*(h) = const$ such that all type P tori for the given energy 
intersect this line. When approaching the left boundary of the P-region 
the torus shrinks to an ellipse with $\cosh \lambda^*(h) = -(m_1 + m_2)/( 2  h) = h_\lambda/h > 1$.
Considering the phase portrait of $H_\lambda$ it is clear that all type P tori
intersect the corresponding window. 
However, the intersection points are not spaced half-way around the torus, and hence 
would give a more complicated description. 
An alternative is to use $p_\lambda = 0$ which does cut the torus twice with equal spacing. 
The physical meaning of this is to assign a symbol to the tangency with the caustic.

\section{Conclusion}

We have   described all syzygy sequences which are realized in the two centre problem.
They are encoded by   Sturmian words as per theorem 1 and 2.

\section{Some Open Problems.} 

{\sc Orbit Counts.} This  work   began  as a warm-up for   understanding  the  syzygy sequences
arising  in the full planar  three-body problems. Does it shed  light
on that problem, or on restricted versions of it?  We  take a peek at this question from
the perspective of orbit counting. 

Let $N(L)$ be the number of periodic orbits of length less than $L$
in a dynamical system. 
For chaotic systems such as Axiom A  systems $N(L)$ grows exponentially with $L$, 
while for integrable systems one expects $N(L)$ to  grow polynomially.  Various theorems
have been established relating $N(L)$ to the  topological entropy of a system. 
In  planar three-body problems  we will instead  let $N(L)$ be    the number of distinct syzygy sequences
of length  $L$  which are suffered by some   periodic orbit during   one period. 
(The original $N(L)$, the  number of orbits themselves is uncountably infinite since the orbits
of a particular topological type,  forms a continuum in an integrable system,
sweeping out all of a torus, or a torus minus collisions.  So we do not want to count
individual orbits.) 
 Being integrable,  we expect that $N(L)$ grows polynomially for the two centre problem. 
 Indeed, $N(L)$ grows like $L^2$.   To estimate  this growth, recall that altogether  P orbits   yield a single    syzygy sequence
$1212 \ldots$.  To count the number of sequences due to L and S orbits,  
 recall that these sequences  are encoded by their rational rotation number $W = p/q$, $W > 0$
 with  the length $L$ of the   sequence being $2 (p+q)$.  So
$N(L)$ is bounded by the number of lattice points $(p,q)$ inside the
diamond $|x|+|y|  \le \frac{1/2} L$, which is $\frac{1}{8} L^2$.  
We   count   both the S and L type, yielding the   bound $N(L) \le \frac{1}{4} L^2 + 1$.

On the other hand, if we use the strong force potential $-m_1/r_1 ^2 -m_2 / r_2 ^2$,  
then one can prove  using variationally methods, that all syzygy sequences are realized in the two centre problem 
except for those with ``stutters'', meaning two of the same type of symbols adjacent to each other.
One estimates   $N(L) \ge 3 * 2^{L-2}$ for the number of such symbols.

The two-centre problem fits  within various one-parameter  families of 2-degree of freedom 
Hamiltonian systems $H_{\lambda} (x,y, p_x, p_y)$. For example, by ``spinning'' the primaries we can fit it into a
one-parameter family, parameterized by the spin rate of the primaries, which at the other endpoint
becomes the  restricted three-body problem.
Or, by changing the potential, we can fit the problem into a one-parameter family which  
includes  strong-force two center problem.   How does $N_{\lambda} (L)$
depend on $\lambda$ for these other problems? 

{\sc Sturn and twist} As remarked above, Sturmian words arise generically in twist maps  \cite{AubryLeDaeron83}.
This  ``Sturmian thread''   connecting the two centre problem to general   twist maps  
gives us   hope that  enough twist-map  structure     persists 
as we  turn on  the  
spin  parameter so that many of the periodic orbits or
quasi-periodic irrational Sturmian type orbits will   persist  all the way to the restricted three body problem.

\section{Acknowledgment}

We  thank  Rick Moeckel (U. Minnesota, US) and Ivan Sterling (St Mary's College, Maryland, US) for accompanying us on
the initial leg of this journey. We  thank Michael Magee (Princeton) for alerting us to the long history of  
``Sturmian words'' going back to  Morse and Hedlund.  
RM   thanks the support of the  NSF DMS1305844  grant
and HD  the  Australian ARC DP110102001 grant. 
HRD would like to thank the UC Santa Cruz Department  of Mathematics for their hospitality.


\bibliographystyle{plain}
\bibliography{../bib_cv/all,../bib_cv/hd}

\end{document}